\documentclass{article}
\usepackage[toc,page]{appendix}
\usepackage{amsmath,amsthm,amsfonts,amssymb}
\usepackage{fullpage}
\usepackage{graphicx}
\usepackage{dsfont} 

\newcommand{\Exp}{\mathbb{E}}
\newcommand{\e}{\mathrm{e}}
\newcommand{\eps}{\varepsilon}
\renewcommand{\P}{\mathbb{P}}
\newcommand{\Q}{\mathbb{Q}}
\newcommand{\Bin}{\mathrm{Bin}}
\usepackage{url}

\usepackage{comment}

\newtheorem{theorem}{Theorem}
\newtheorem{lemma}{Lemma}

\newtheorem{definition}{Definition}

\newtheorem{conjecture}{Conjecture}

\usepackage{color}

\usepackage{enumitem}
\setlist[enumerate,1]{label=(\arabic*),noitemsep}

\usepackage{xspace,prettyref}

\newcommand{\iiddistr}{{\stackrel{\text{\iid}}{\sim}}}

\newcommand{\reals}{{\mathbb{R}}}

\newcommand{\naturals}{{\mathbb{N}}}

\newcommand{\supp}{\mathrm{supp}}

\newcommand{\identity}{\mathds I}
\newcommand{\allones}{\mathds J}

\newcommand{\diff}{{\rm d}}

\newcommand{\expect}[1]{\Exp\!\left[ #1 \right]}

\newcommand{\prob}[1]{ \mathbb{P}\left\{ #1 \right\} }

\newcommand{\Binom}{\mathrm{Bin}}

\newcommand{\Hyp}{\mathrm{Hyp}}
\newcommand{\eg}{e.g.\xspace}
\newcommand{\ie}{i.e.\xspace}
\newcommand{\iid}{i.i.d.\xspace}
\newrefformat{eq}{(\ref{#1})}
\newrefformat{chap}{Chapter~\ref{#1}}
\newrefformat{sec}{Section~\ref{#1}}
\newrefformat{alg}{Algorithm~\ref{#1}}
\newrefformat{fig}{Fig.~\ref{#1}}
\newrefformat{tab}{Table~\ref{#1}}
\newrefformat{rmk}{Remark~\ref{#1}}
\newrefformat{clm}{Claim~\ref{#1}}
\newrefformat{def}{Definition~\ref{#1}}
\newrefformat{cor}{Corollary~\ref{#1}}
\newrefformat{lmm}{Lemma~\ref{#1}}
\newrefformat{prop}{Proposition~\ref{#1}}
\newrefformat{app}{Appendix~\ref{#1}}
\newrefformat{hyp}{Hypothesis~\ref{#1}}
\newrefformat{thm}{Theorem~\ref{#1}}
\newrefformat{ass}{Assumption~\ref{#1}}
\newrefformat{conj}{Conjecture~\ref{#1}}

\newcommand{\norm}[1]{\left\|{#1} \right\|}

\newcommand{\inner}[2]{\langle #1, #2 \rangle}
\newcommand{\indc}[1]{{\mathbf{1}_{\left\{{#1}\right\}}}}

\newcommand{\calD}{{\mathcal{D}}}

\newcommand{\calF}{{\mathcal{F}}}

\newcommand{\calN}{{\mathcal{N}}}

\newcommand{\calT}{{\mathcal{T}}}

\newcommand{\calV}{{\mathcal{V}}}
\newcommand{\calW}{{\mathcal{W}}}

\DeclareMathAlphabet{\varmathbb}{U}{bbold}{m}{n}

\newcommand{\argmax}{\mathrm{argmax}}
\newcommand{\dX}{\mathrm{d}X}

\renewcommand{\hat}{\widehat}
\renewcommand{\tilde}{\widetilde}
\newcommand{\contig}{\trianglelefteq}

\usepackage{bbm}

\newcommand{\R}{\mathbb{R}}

\newcommand{\J}{\mathbb{J}}
\DeclareMathOperator{\Tr}{Tr}

\newcommand{\bm}{\boldsymbol}

\newcommand{\planted}{\sigma_{0}}
\newcommand{\score}{\calT}

\newcommand{\up}{\textup{upper}}
\newcommand{\low}{\textup{lower}}

\newcommand{\indicator}[1]{\bm{1}_{#1}}

\newcommand{\MMSE}{{\rm MMSE}}
\newcommand{\snr}{{\mathsf{snr}}}
\newcommand{\Keywords}[1]{\par\noindent 
{\small{\em Keywords\/}: #1}}
\newcommand{\ER}{Erd\H{o}s-R\'{e}nyi\,}

\begin{document}

\title{Information-theoretic bounds and phase transitions in clustering, sparse PCA, and submatrix localization}
\author{Jess Banks \and Cristopher Moore \and Roman Vershynin  \and Nicolas Verzelen \and Jiaming Xu
\thanks{J. Banks is with the Department of Mathematics, University of California, Berkeley, Berkeley, CA, \texttt{jess.m.banks@berkeley.edu}.
C. Moore is with the Santa Fe Institute, Santa Fe, NM, \texttt{moore@santafe.edu}.
R.~Vershynin is with the Department of Mathematics, University of Michigan, Ann Arbor, MI, \texttt{romanv@umich.edu}. 
N. Verzelen is with UMR 729 MISTEA, INRA, Montpellier.
J.~Xu is with the Krannert School of Management, Purdue University, West Lafayette, IN 47907,
\texttt{xu972@purdue.edu}.  
This work was in part presented at the 54th Annual Allerton Conference on Communication, Control, and Computing, 
September 27-30, 2016, Monticello, IL, USA.
}
}

\maketitle

\begin{abstract}
	We study the problem of detecting a structured, low-rank signal matrix corrupted with additive Gaussian noise.  This includes clustering in a Gaussian mixture model, sparse PCA, and submatrix localization.  Each of these problems is conjectured to exhibit a sharp information-theoretic threshold, below which the signal is too weak for any algorithm to detect.  We derive upper and lower bounds on these thresholds by applying the first and second moment methods to the likelihood ratio between these ``planted models'' and null models where the signal matrix is zero.  
For sparse PCA  and  submatrix localization, we determine this threshold exactly in the limit where the number of blocks is large or the signal matrix is very sparse; for the clustering problem, our bounds differ by a factor of $\sqrt{2}$ when the number of clusters is large.
	Moreover, our upper bounds show that for each of these problems there is a significant regime where reliable detection is information-theoretically possible but where  known algorithms such as PCA fail completely, since the spectrum of the observed matrix is uninformative.  This regime is analogous to the conjectured `hard but detectable' regime for community detection in sparse graphs.

\end{abstract}
\Keywords{First and second moment methods, clustering, information-theoretic bounds, sparse PCA, submatrix localization }

\section{Introduction}
\label{sec:intro}

Many problems in machine learning, signal processing, and statistical inference have a common, unifying goal: reconstruct a low-rank signal matrix observed through a noisy channel. This framework can encompass a wide range of tasks as we vary the channel and low-rank signal, but we focus here on the case where the noise is additive and Gaussian, and the signal is relatively weak in comparison to the noise.  
To be precise, suppose we are given an $m \times n$ data matrix
\begin{align}
	X = M + W  
	\quad \text{with} \quad
	M = \frac{\snr}{\sqrt{n} } U V^\dagger \, ,
	\label{eq:planted_model}
\end{align}
where $\snr$ is a fixed parameter characterizing the signal-to-noise ratio, 
$U \in \reals^{m \times k}$ and $V \in \reals^{n \times k}$ are generated from some known prior distribution independent of $n$, 
and $W \in \reals^{m \times n}$ is a noise matrix whose entries are independent Gaussians with unit variance. We will refer to this as the \emph{planted model}: it consists of a noisy observation $X$ of a signal matrix $M$ of rank $k$, and may possess additional structure through the priors on $U$ and $V$.

Given the observed matrix $X$, the problem of interest is to reconstruct $M$, or at least detect that it exists. For simplicity, we will work in the Bayes-optimal case where model parameters such as the true rank and signal-to-noise ratio are known to the estimators. In the low signal-to-noise ratio regime we consider, exact reconstruction of $M$ is fundamentally impossible (see \S2 for more details).  Instead, we focus on the following two tasks: first, detecting that the signal $M$ exists, i.e., telling with high probability whether $X$ was indeed generated by the planted model as opposed to a \emph{null model} where $M=0$ and $X$ consists only of noise; and second, reconstructing $M$ to some accuracy better than chance.  We define these tasks formally as follows.

\begin{definition}[Detection]
Let $\P(X)$ be the distribution of $X$ in the planted model~\prettyref{eq:planted_model}, and denote by $\Q(X)$ the distribution of $X$ in the null model where $X=W$.  A test statistic $\calT(X)$ with a threshold $\epsilon$ achieves detection if $\lim_{n \to \infty} \left[ \P( \calT(X) < \epsilon) + \Q(\calT(X) \ge \tau) \right] = 0$, so that the criterion $\calT(X) \ge \epsilon$ determines with high probability whether $X$ is drawn from $\P$ or $\Q$.
\end{definition}

\begin{definition}[Reconstruction]
An estimator $\hat{M} = \hat{M}(X)$ achieves reconstruction if 
$ \Exp_{X} \| \hat{M}\|_F^2 = O(n)$ and 
there exists a constant $\epsilon>0$ such that 
$\lim_{n \to \infty} (1/n) \Exp_{M,X} \inner{M}{\hat{M}} \ge \epsilon$,
where $\inner{A}{B} = \Tr A^\dagger B$ denotes the matrix inner product and $\|A\|_F^2=\inner{A}{A}.$
\end{definition}

\noindent 
For many natural problems in this class, it is believed that there is a phase transition, i.e., a threshold value of $\snr$ below which both tasks are information-theoretically impossible: no test statistic can distinguish the null and planted models, and no estimator can beat the trivial one $\hat{M}=0$.  This threshold is known as the \emph{information-theoretic} threshold and it also depends on the structure of the problem, i.e., on the priors of $U$ and $V$; if this prior is more strongly structured, we expect the threshold to be lower.

We focus on three cases of~\prettyref{eq:planted_model} which arise in many applications.  In \emph{Sparse PCA}, $k=1$ and $U=V=v$ for some vector $v$.  We further assume that $v$ is sparse, with a constant fraction of nonzero entries. This corresponds to the \emph{sparse, spiked} Wigner model of~\cite{MontanariPCA14,lesieur2015phase}.
In \emph{Submatrix Localization} (also known as submatrix detection and noisy clustering), 
$U=V$ and $M$ contains $k \ge 2$ distinct blocks of elevated mean. This model arises in the analysis of social networks and gene expression, see \eg 
\cite{kolar2011submatrix,butucea2013,HajekWuXu_MP_submat15}; 
it can also be thought of as a Gaussian version of the stochastic block model~\cite{decelle-etal-a,decelle-etal-b}. Finally, in \emph{Gaussian Mixture Clustering}, there are $k \ge 2$ clusters, and each row of $M$ is the center of the cluster to which the corresponding data point belongs. This model has been widely studied, see, \eg,~\cite{Vempala04,Srebro06,KannanVempala09,kannan2008spectral,achlioptas2005spectral}. 

For each of these three problems, our goal is to compute the information-theoretic threshold, and understand how it scales with the parameters of the problem: for instance, the sparsity of the underlying signal or the number of clusters.  In particular, a simple upper bound on the information-theoretic threshold for each of these problems is the point at which spectral algorithms succeed, i.e., the point at which the likely spectrum of $X$ becomes distinguishable from the spectrum of the random matrix $W$.  The spectral thresholds for our problems are well known from the theory of Gaussian matrices with low-rank perturbations.  However, based on compelling but non-rigorous arguments from statistical physics (e.g.~\cite{lesieur2015mmse,clustering}), it has been conjectured that when the signal is sufficiently sparse, or its rank (the number of clusters or blocks) is sufficiently large, the information-theoretic threshold falls strictly below the spectral one.


In this paper, we prove upper and lower bounds on information-theoretic thresholds of all three problems, determining the threshold within a multiplicative constant in interesting regimes.  
For sparse PCA, we determine the precise threshold in the limit where the signal matrix is very sparse; similarly, for the submatrix localization problem, we determine the threshold when the number of blocks is large.  For the clustering problem, our bounds differ by a factor of $\sqrt{2}$ in the limit where the number of clusters is large.
Moreover, our results verify the conjecture that the information-theoretic threshold dips below the spectral one when the signal is sufficiently sparse, or when the number of clusters or blocks is sufficiently large.  This corresponds to recent results~\cite{AbbeSandon16,banks-etal-colt,ChenXu14,HajekWuXu_one_info_lim15} showing that, in the stochastic block model, the information-theoretic detectability threshold falls below the Kesten-Stigum bound above which efficient spectral and message-passing algorithms succeed~\cite{decelle-etal-a,decelle-etal-b,mossel-neeman-sly-colt,krzakala-etal-nonbacktracking,bordenave-lelarge-massoulie}.  We consider this evidence for the conjecture that these problems posses a `hard but detectable' regime where detection and reconstruction are information-theoretically possible but take at least exponential time. Although our computations are specific to these models, our proof techniques are quite general and may be applied with mild adjustment to a broad range of similar problems. We present our results in \S2, an overview of our proof techniques in \S3, and full proofs in in \S4. 

Since the initial posting of this paper as an arXiv preprint, a number of interesting papers~\cite{PerryWeinBandeiraMoitra16,PerryWeinBandeira16,LelargeMiolane16} have appeared, some
extending or improving our results.  Sharp lower bounds for sparse PCA were also obtained recently in~\cite{PerryWeinBandeira16} using a conditional second moment method similar to ours. 
Complete, but not explicit, characterizations of information-theoretic reconstruction thresholds were obtained in~\cite{KrzakalaXuZdeborova16,LelargeMiolane16} for sparse PCA and submatrix localization through the Guerra interpolation technique and cavity method.  However, their characterization of reconstruction thresholds does not directly apply to detection.


\section{Models and results}

\subsection{Sparse PCA}
\label{sec:sparsepca}
Consider the following \emph{spiked Wigner} model, where the underlying signal is a rank-one matrix:
\begin{align} \label{eq:sparse_pca_def}
	X = \frac{\lambda}{\sqrt n} vv^\dagger + W \, ,
\end{align}
Here, $v \in \R^n$, $\lambda > 0$ and $W \in \R^{n \times n}$ is a Wigner random matrix with $W_{ii} \iiddistr \calN(0, 2)$ 
and $W_{ij}=W_{ij} \iiddistr \calN(0, 1)$ for $i<j$.  We assume $v$ is drawn from the sparse Rademacher prior, although many alternatives may be imposed.  Specifically, for some $\gamma \in [0,1]$ the support of $v$ is drawn uniformly from all ${n \choose \gamma n}$ subsets $S \subset [n]$ with $|S| = \gamma n$ (when $n$ is finite, we assume that $\gamma n$ is an integer).  Once the support is chosen, each nonzero component $v_i$ is drawn independently and uniformly from $\{\pm \gamma^{-1/2} \}$, so that 
$\|v\|_2^2 = n$.   When $\gamma$ is small, the data matrix $X$ is a sparse, rank-one matrix observed through Gaussian noise. 

One natural approach for this problem is PCA: that is, diagonalize $X$ and use its leading eigenvector $\hat{v}$ as an estimate of $v$.  The threshold at which this algorithm succeeds can be computed using the theory of random matrices with rank-one perturbations~\cite{baik2005phase,peche2006largest,benaych2011eigenvalues}:
\begin{enumerate}
	\item When $\lambda > 1$, the leading eigenvalue of $X/\sqrt{n}$ converges almost surely to $\lambda + \lambda^{-1}$, and $\inner{v}{\hat{v}}^2$ converges almost surely to $1-\lambda^{-2}$; thus PCA succeeds in reconstructing better than chance; 

	\item When $\lambda  \le 1$, the leading eigenvalue of $X/\sqrt{n}$ converges almost surely to $2$, and $\inner{v}{\hat{v}}^2$ converges almost surely to $0$; thus PCA fails to reconstruct better than chance.
\end{enumerate}

Because the leading eigenvalue of $W$ is 2 w.h.p., detection is only possible when $\lambda > 1$. Intuitively, PCA only exploits the low-rank structure of the underlying signal, and not the sparsity of $v$; it is natural to ask whether one can succeed in detection or reconstruction for some $\lambda < 1$ by taking advantage of this additional structure. 
Through analysis of an approximate message-passing algorithm and the free energy, the  following conjecture was made in statistical physics~\cite{lesieur2015phase,KrzakalaXuZdeborova16}:

\begin{conjecture}\label{conj:sparse_pca}
	Let the computational threshold be the minimum of $\lambda$ so that reconstruction or detection can be attained in polynomial-time in $n$ for a given $\gamma.$ There exists $\gamma^\ast \in (0,1)$ such that 
\begin{enumerate}
	\item If $\gamma \ge \gamma^\ast$, then  both the information-theoretic and computational thresholds are given by $\lambda = 1$.
	\item If $\gamma < \gamma^\ast$, then the computational threshold is given by $\lambda = 1$, but the information-theoretic threshold for $\lambda$ is strictly smaller.  
\end{enumerate}
\end{conjecture}

\noindent We derive the following upper and lower bounds on the information-theoretic threshold in terms of $\lambda$ and $\gamma$, and 
confirm that the threshold is $\lambda=1$ when $\gamma$ is relatively large and falls strictly below $\lambda = 1$ when $\gamma$ is sufficiently small.  Throughout, we use 
\[ 
h(\gamma) = -\gamma \log \gamma - (1-\gamma) \log (1-\gamma)
\]
to denote the entropy function, and $\calW(y)$ for the root $x$ of $x \e^x = y$. All our logarithms are natural.  
\begin{theorem}
\label{thm:sparse-pca}
Let
	\begin{align}
		\lambda^{\up} &= 2 \sqrt{h(\gamma) + \gamma \log 2} 
		\label{eq:sparse-pca-up} \\
		\lambda^{\low} 
		&=\begin{cases}
			1 & \gamma \ge 0.6 \\
			\sqrt{ 2 \gamma \,\calW\!\left( \frac{1}{2 \gamma\sqrt{\e}} \right)} & \e^{-41}/81 \le \gamma < 0.6 \\
			\sqrt{ 4 \gamma\left(- \log \gamma - 2.1\sqrt{-2 \log \gamma} - \frac{3}{2}\log\left(\frac{3\e}{1-\gamma}\right) \right)} & \gamma < \e^{-41}/81 \, .
		\end{cases}
		\label{eq:sparse-pca-low} 
	\end{align}
Then detection and reconstruction are information-theoretically possible when $\lambda > \lambda^{\up}$ and are impossible when $\lambda < \lambda^{\low}$. 
\end{theorem}

\noindent In our proof, we give tighter lower bounds, but these are analytically convenient. 

Note that $\lambda^{\up}$ falls below the spectral threshold $\lambda=1$ whenever $\gamma \le 0.054$, 
and $\lambda^{\low}$ matches the spectral threshold whenever $\gamma \ge 0.6$. Hence,
Theorem~\ref{thm:sparse-pca} proves Conjecture~\ref{conj:sparse_pca} on information-theoretic threshold, 
albeit without pinning down $\gamma^*$ exactly.
In addition, in the limit $\gamma \to 0$, both $\lambda^{\up}$ and $\lambda^{\low}$ give an information-theoretic threshold of
\begin{equation}
\label{eq:pca-tight}
\lambda_c = 2 \left(1+o_\gamma(1) \right) \sqrt{ - \gamma \log \gamma } \, , 
\end{equation}
determining the threshold fully in the limit where the low-rank matrix is very sparse.  
Independent of the present work, and building on our preprint~\cite{Banks16}, Perry et al.~\cite{PerryWeinBandeira16} 
obtained the same tight threshold in this limit with a smaller error term. 
Previous work~\cite{Cai2015} had determined that threshold scales as $\lambda = \Theta(\sqrt{ - \gamma \log \gamma})$ up to a constant factor.

In passing, we note that there is a very interesting line of work on  \emph{exact or approximate support reconstruction} for sparse PCA, \ie, estimating correctly or consistently the positions of non-zeros in $v$, in a regime where the size of the support is sublinear in $n$ (see \eg,~\cite{JohnstoneLu09,amini2009,berthet2013optimal,krauthgamer2015,deshpande2014sparse} and references therein)\footnote{To be precise, those references mostly study the spiked Wishart model: $X= (\lambda/\sqrt{n}) u v^\top +W$, where $u  \in \R^m$ with i.i.d. $\calN(0,1)$ entries and $W$ is a $m\times n$ Gaussian matrix with i.i.d. $\calN(0,1)$ entries; the results can be readily extended to the spiked Wigner model.}. 
In an influential paper~\cite{JohnstoneLu09}, it was shown that while the estimate via the classical PCA is inconsistent, a simple diagonal thresholding procedure consistently estimates $v$ provided that $v$ is sufficiently sparse.  Assuming $\lambda=\Theta(1)$, it is later proved in~\cite{amini2009} that diagonal thresholding exactly recovers the support of $v$ with high probability if $ \gamma \lesssim 1/ (n\log n)$, and the information-theoretic threshold for exact support recovery is given by $\gamma \asymp 1/\log n$ up to a constant factor. 
In contrast, we focus on the regime where the size of the support
is linear in $n$, \ie, $\gamma=\Theta(1)$, and $\lambda=\Theta(1)$. 
In this regime it is impossible to correctly or 
consistently estimate the support of $v$, and hence 
we instead focus on detection and reconstruction better than chance.

\subsection{Submatrix Localization}

In the submatrix localization problem, our task is to detect within a large Gaussian matrix a small block or blocks with atypical mean.  
Let $\sigma : [n] \to [k]$ be a \emph{balanced} partition, \ie one for which $|\sigma^{-1}(t)| = n/k$ for all $t\in[k]$, chosen uniformly from all such partitions. This terminology will recur throught the paper.
Construct a 
$n \times n$ matrix $Y$ such that $Y_{i,j} = \indicator{\sigma(i) = \sigma(j)}$. In the planted model, 
\begin{align} \label{eq:submatrix_def}
	X = \frac{\mu}{\sqrt n}\left(Y - \frac{1}{k} \allones \right) + W \, ,
\end{align}
where $W$ is again a Wigner matrix and $\allones$ is the all-ones matrix.  In the null model, $X=W$. The subtraction of $ \allones/k$  centers the signal matrix so that $\Exp X = 0$ in both the null and planted models. 
In the planted model, $ (\mu/\sqrt n)\left(Y - \allones/k \right)$ is a rank-$(k-1)$ matrix with the largest $(k-1)$ eigenvalues all equal to $\mu \sqrt{n}/k$, making $X$ a Wigner matrix with a rank-$(k-1)$ additive perturbation. Matrices of this type exhibit the following spectral phase transition:

\begin{enumerate}
	\item When $\mu > k$, the $k$ leading eigenvalues of $X/\sqrt{n}$ converge to $ \mu/k + k/\mu$ almost surely; 
	\item When $\mu \le k$, the $k$ leading eigenvalues of $X/\sqrt{n}$ converge to $2$ almost surely. 
\end{enumerate}

\noindent 
Hence, it is possible to detect the presence of the additive perturbation from the spectrum of $X$ alone when $\mu>k$. We prove the following upper and lower bounds on the information-theoretic threshold:
\begin{theorem} \label{thm:submatrix}
	Let
	\begin{align}
		\mu^{\up} &= 2 k  \sqrt{ \frac{ \log k}{k-1} } \\
		\mu^{\low} &= \begin{cases}
		2 & k=2 \\
		 k \sqrt{ \frac{2  \log (k-1)}{k-1} }  &  3 \le k \le \exp(22^4) \\	 
		 2\sqrt{k\log k-11 k\log^{3/4}(k)}  & k > \exp(22^4) \,.	 
		 \end{cases}
	\end{align}
	Then detection and reconstruction are information-theoretically possible when $\mu > \mu^{\up}$ and impossible when $\mu < \mu^{\low}$.
\end{theorem}
Note that $\mu^{\up}$ dips below the spectral threshold $ \mu=k$ when $k \ge  11$, indicating a regime where standard spectral methods fail but detection is information-theoretically possible. 
Also, \prettyref{thm:submatrix} proves the conjecture  in~\cite{lesieur2015mmse} that as $k \to \infty$, the information-theoretic threshold is given by $\mu= 2 \sqrt{ k \log k}$.


Previous work in  \emph{submatrix detection} and \emph{localization}, also known as \emph{noisy biclustering}, 
mostly focuses on finding a single submatrix, see,  \eg,
\cite{kolar2011submatrix,butucea2013,HajekWuXu_MP_submat15}
and the references therein. 
Notably, ~\cite{ChenXu14} considers a general setting where the number of blocks $k$ could grow with $n$,
and proves that if $\mu\ge c \sqrt{k \log n}$ for some large constant $c$, then it is informationally possible 
to exactly reconstruct the support of the planted submatrices with high probability; if $\mu \le c' \sqrt{ k \log n}$,
then exact support reconstruction is informationally impossible. In our setting, $\mu=\Theta(1)$ and $k=\Theta(1)$,
so it is impossible to consistently estimate the support and we instead resort to detection and reconstruction better than
the chance.

\subsection{Gaussian Mixture Clustering}

Finally, we study a model of clustering with limited data in high dimension. Let $v_1,...,v_k$ be independently and identically distributed as $\calN \left(0, k/(k-1) \, \identity_{n,n} \right),$ and define $\overline v = (1/k)\sum_{s} v_s$ to be their mean. The scaling of the expected norm of each $v_s$ with $k$ ensures that $\Exp \|v_s - \overline{v} \|_2^2 =n$ for all $1 \le s \le k$. For a fixed parameter $\alpha>0$, we then generate $m = \alpha n$ points $x_i \in \reals^n$ which are partitioned into $k$ clusters of equal size by a balanced partition $\sigma:[n]\to[k]$, again chosen uniformly at random from all such partitions. For each data point $i$, let $\sigma_i \in [k]$ denote its cluster index, and generate $x_i$ independently according to Gaussian distribution with mean $\sqrt{\rho/n}\, (v_{\sigma_i} - \overline v)$ and identity covariance matrix, where  $\rho>0$ is a fixed parameter characterizing the cluster separation. 
We can put this in the form of model \eqref{eq:planted_model} by constructing an $ n \times k$ matrix $V = [v_1, ..., v_k ]$, an $m \times k$ matrix $S$ with $S_{i,t} = \indicator{\sigma_i = t}$,  and setting
\begin{align} \label{eq:clustering_def}
	X = \sqrt{\frac{\rho}{ n}} \left(S - \frac{1}{k}\allones_{m,k}\right) V^\dagger + W,
\end{align}
where $W_{i,j} \iiddistr \calN(0,1).$
  In the null model, there is no cluster structure and $X = W$. 
The subtraction of $\allones_{m,k}/k$ once again centers the signal matrix so that $\Exp X = 0$ in both models. 
The following spectral phase transition follows from the celebrated BBP phase transition~\cite{baik2005phase,Paul07}:
\begin{enumerate}
\item When $ \rho \sqrt{\alpha} >k-1$, then the largest eigenvalue of  $(1/m) X^{\dagger} X$
converges to $( 1+ \frac{\rho}{k-1}) ( 1+  \frac{k-1}{ \rho \alpha })$ almost surely;
\item When $  \rho \sqrt{\alpha}  \le k-1$, then the largest eigenvalue of  $(1/m) X^{\dagger} X$ converges to $ (1+ 1/\sqrt{\alpha} )^2$  almost surely. 
\end{enumerate}
%
Thus spectral detection  is possible if  $\rho \sqrt{\alpha} >(k-1) $. 

We prove the following upper and lower bounds on the information-theoretic threshold, which differ by a factor of $\sqrt{2}$ when $k$ is large.

\begin{theorem}
	  Let 
	\begin{align}
		\rho^\up &= 2\sqrt{ \frac{k \log k}{\alpha} } + 2\log k \\
		\rho^\low &= \begin{cases}
		              \sqrt{1/\alpha} & k=2 \\
		              \sqrt{\frac{2(k-1)\log(k-1)}{\alpha}} & k\geq 3 \,.
		             \end{cases}
	\end{align}
	Then detection and reconstruction are possible when $\rho > \rho^\up$ and impossible when $\rho< \rho^\low$. 
\end{theorem}
%
We conjecture that in the limit $k \to \infty$, the information-theoretic threshold is $ \rho = 2 \sqrt{k\log k /\alpha}$, but we do not find a proof.

Most previous work in Gaussian mixture clustering focuses exact or near-exact reconstruction, see \eg,~\cite{Srebro06,KannanVempala09}. As in the \S2.1, a popular approach is PCA, which identifies the clusters based on the first $k$ principal components of the data matrix. It is shown in~\cite{Vempala04} that PCA allows identification of clusters with a cluster separation $\sqrt{\rho}=\Omega(k^{1/4} \log n )$ and a sample of size $m=\Omega(d^3 \log d)$. This technique was extended to non-Gaussian distributions~\cite{kannan2008spectral,achlioptas2005spectral}, and for Gaussian, the cluster separation is improved to be $\sqrt{\rho}=\Omega(\sqrt{k \log n})$ and the sample complexity reduced to $m=\Omega( k^2 d)$ in~\cite{achlioptas2005spectral}. In our setting, since $\rho$ is a fixed constant, the cluster separation is not sufficient for exact reconstruction and we turn to detection and reconstruction better than chance. Somewhat surprisingly, we find that if the number of clusters is large, clustering is informationally possible even when below the spectral phase transition threshold, and we conjecture that in this regime it is computationally hard to identify the clusters. We note that a similar ``hard-but-detectable'' regime has been determined empirically in~\cite{Srebro06}.

\section{Proof techniques}
\label{sec:sketch}

This section gives a brief overview of our proof techniques; the full proof will be presented in the next section. 
\subsection{The likelihood ratio and hypothesis testing}
\label{sec:first}

Detection is a classic hypothesis testing problem.  Given a test statistic $\calT(X)$, we consider its distribution under the planted and null models.  If these two distributions are asymptotically disjoint, i.e., their total variation distance tends to $1$ in the limit of large datasets, then it is information-theoretically possible to distinguish the two models with high probability by measuring $\calT(X)$.  A classic choice of statistic for binary hypothesis testing
is the likelihood ratio, 
\[
	\frac{\P(X)}{\Q(X)} 
	= \frac{\sum_M \P(X, M)}{\Q(X)} 
	= \frac{\sum_M \P(X | M) \,\P(M)}{\Q(X)} \, . 
\]
This object will figure heavily in both our upper and lower bounds. 
By the Neyman-Pearson lemma~\cite{NeymanPearson}, the likelihood ratio is the most powerful test statistic of a given significance level; that is, if we set the threshold $\epsilon$ at the point such that $\Q(\calT(X) < \epsilon) = \alpha$ for any fixed $\alpha$, it maximizes $\P(\calT(X) \ge \epsilon)$.  

Our upper bounds do not use the likelihood ratio directly, since it is hard to furnish lower bounds on the typical value of $\P(X)/\Q(X)$ when $X$ is drawn from $\P$.  Instead, we use the generalized likelihood ratio,
\[
	\max_M \frac{\P(X | M)}{\Q(X)} \, .
\]
In the planted model where the true underlying signal matrix is $M_0$, this quantity is trivially bounded below by $\P(X | M_0)/\Q(X)$.  We will use simple first moment arguments to show that, with high probability in the null model $\Q$, this lower bound is not achieved by any $M$.  An easy extension of this argument shows that, in the planted model, the maximum likelihood estimator (MLE) $\widehat{M} = \argmax_M \P(X | M)$ has nonzero correlation with $M_0$.  Thus we can output a good estimate of $M_0$ by exhaustive search.   

The conditional likelihood ratio has a particularly elegant form when the noise is additive and Gaussian. It is first useful to write down the probability distribution of a Wigner random matrix in the space of symmetric matrices (also known as the Gaussian Orthogonal Ensemble) as
\begin{equation}
	\Q(W) = \frac{1}{Z_n}\e^{-\frac{1}{4}\norm{W}^2_F}
\end{equation}
where $Z_n$ is a normalization constant depending only on $n$. Similarly, if $W$ is a Gaussian random matrix whose entries are independently distribution as $\calN(0,1)$, then
\begin{equation}
	\Q(W) = \frac{1}{Z'_n}\e^{-\frac{1}{2}\norm{W}^2_F}.		
\end{equation}
Thus, for Wigner noise the conditional likelihood ratio is
\begin{equation} 
	\label{eq:conditional-ratio}
	\frac{\P(X|M)}{\Q(X)} = \frac{\Q(X - M)}{\Q(X)} = \e^{\frac{1}{4}\left(\norm{X}^2_F - \norm{X-M}^2_F\right)} = \e^{\frac{1}{2}\inner{X}{M} - \frac{1}{4}\norm{M}^2_F} 
\end{equation}
and identical for Gaussian noise except for a factor of two in the exponent.

The conditional \emph{log} likelihood is a weighted sum of the entries of $X$, and is therefore itself Gaussian. Our first moment bounds now proceed as follows. In the planted model, when $X = M_0 + W$, the conditional log likelihood of the planted signal $M_0$ is $\frac{1}{2}\inner{W}{M_0} + \frac{1}{4}\norm{M_0}^2_F$. We can use standard Gaussian concentration results to bound the typical deviation of this quantity around its mean of $\frac{1}{4}\norm{M_0}^2_F$. On the other hand, in the null model where $X = W$, the conditional log likelihood of any $M$ is $\frac{1}{2}\inner{W}{M} - \frac{1}{4}\norm{M}^2_F$, and we can combine Gaussian concentration bounds with a union bound over all possible $M$ to compute when, with high probability in the null model, \emph{no} $M$ will beat the conditional log likelihood of the planted signal $M_0$. 

As above, this argument can be nominally modified to show that the MLE estimate is positively correlated with $M_0$. This technique must be altered slightly for the case of Gaussian Mixture Clustering, since $M \propto SV^\dagger$ and $V$ is chosen from a continuous prior. 
We could take a union bound over $V$ by suitably discretizing; instead we use the fact that for any $S$ we can analytically maximize the conditional likelihood of $SV^\dagger$ with respect to $V$ by setting the $i$th cluster center to be the empirical center of the data points assigned to cluster $i$ by $S$. This alters the distribution of the likelihood ratio, but we proceed analogously to above with the proper concentration results.

\subsection{Second moment bounds and contiguity} 
\label{sec:contig}

Intuitively, if the planted model $\P$ and the null model $\Q$ have asymptotically disjoint support, then the likelihood ratio $\P/\Q$ is almost always either very large or very small.  In particular, its variance in $\Q$  must diverge.  This suggests that we can derive lower bounds on the threshold by bounding its second moment in $\Q$, or equivalently its expectation in $\P$.  Suppose the second moment is bounded by some constant $C$, \ie,
\begin{equation}
\label{eq:smm-bound}
\Exp_{X \sim \Q} \!\left[ \left( \frac{\P(X)}{\Q(X)} \right)^{\!2} \right]
= \Exp_{X \sim \P} \!\left[ \frac{\P(X)}{\Q(X)} \right]
= \int_X \dX \,\frac{\P(X)^2}{\Q(X)} 
\le C \, .
\end{equation}
This implies a bound on the Kullback-Leibler divergence between $\P$ and $\Q$, since Jensen's inequality gives
\begin{align}
D_{\mathrm{KL}} (\P \| \Q) 
= \Exp_{X \sim \P} \log \frac{\P(X)}{\Q(X)} 
\le \log \Exp_{X \sim \P} \frac{\P(X)}{\Q(X)} 
\le \log C = O(1) \, . \label{eq:secondmoment_KL}
\end{align}
Moreover, it also implies that detection is impossible. To see this, consider the following definition:
\begin{definition}
Let $\P=(\P_n), \Q=(\Q_n)$ be sequences of distributions defined on the same sequence of spaces $\Omega_n$.  We write $\P \contig \Q$, and say that \emph{$\P$ is contiguous to $\Q$}, if for any sequence of events $E=(E_n)$ such that $\Q(E) \to 0$, we also have $\P(E) \to 0$.
\end{definition}
\noindent 
If $\P \contig \Q$ then detection is impossible, since no algorithm can return ``yes'' with high probability (or even positive probability) in the planted model, and ``no'' with high probability in the null model. 

The following simple argument shows that~\prettyref{eq:smm-bound} implies $\P \contig \Q$ and hence non-detectability. Let $E$ be a sequence of events such that $\Q(E) \to 0$, and let $\indicator{E}$ denote the indicator random variable for $E$.  Then Cauchy-Schwarz gives
\begin{equation} \label{eq: contiguity}
	\P(E) 
	= \Exp_{X \sim \P} \,\indicator{E}
	= \Exp_{X \sim \Q} \frac{\P(X)}{\Q(X)} \,\indicator{E} 
	\le \sqrt{
	\Exp_{X \sim \Q}  \left( \frac{\P(X)}{\Q(X)} \right)^{\!2}  
	\times \Exp_{X \sim \Q} \,\indicator{E}^2 
	}
	\le \sqrt{C \Q(E)} \to 0 \, .
\end{equation}
We note that showing that $\Q \contig \P$ often requires additional arguments such as the small subgraph conditioning method~\cite{mossel-neeman-sly,banks-etal-colt}.


The following lemma gives a general expression for the second moment of the likelihood ratio whenever the model consists of a symmetric signal matrix with Wigner noise or an asymmetric signal matrix with Gaussian noise.
\begin{lemma} \label{lmm:secondmomentGaussian}
	Let $\P(X)$ and $\Q(X)$ be the planted and null models $X = M + W$ and $X = W$ respectively, where $M$ is drawn from some prior over symmetric matrices and $W$ is a Wigner matrix.  Then
	\[
	\Exp_{X \sim \Q} \,\left(\frac{\P(X)}{\Q(X)}\right)^2 = \Exp_{M, M'} \e^{ \frac{1}{2} \inner{M}{M'} } \, , 
	\]
	where $M$ and $M'$ are drawn independently from the prior.  Similarly, if  $M$ is drawn from some prior over asymmetric matrices and $W$ is a Gaussian random matrix, 
	\[
	\Exp_{X \sim \Q} \,\left(\frac{\P(X)}{\Q(X)}\right)^2 = \Exp_{M, M'} \e^{ \inner{M}{M'} } \, . 
	\]
\end{lemma}
Thus the second moment method boils down to calculating an exponential moment of the correlation between two independent draws from the prior $\P(M)$.
 Depending on $\P(M)$, these draws can correlate in complicated ways, and the remainder of our second moment computations consist of combinatorially analyzing the various events in $\P(M)$ that give rise to these correlations.

\subsection{Conditional second moment method}

Sometimes, rare events can cause the second moment to explode even when two models are truly contiguous. We circumvent this by computing the second moment conditioned on a sequence of high-probability events $F$ 
which rule out the catastrophic rare ones.

In the planted model, these events can occur both in the prior distribution of the planted signal, and in the additive Gaussian noise of the channel. To address atypical events in the prior, we will condition on some high probability property of the signal $M$, choosing an appropriate event $\{M \in F\}$ such that $\P(F) = 1 + o(1)$, and form the corresponding conditional distribution
\begin{align}
	\P'(X) = \frac{\Exp_M[ \P(X |M)\indicator{F}]}{\P(F)}.  \label{eq:conditional_distribution_def_1}
\end{align}
The case where the problematic events occur in the noise distribution is similar; we define an event $\{X - M \in F_M\}$ such that $\P(F_M | M) = 1 + o(1)$ uniformly over $M$, and the corresponding conditional distribution
\begin{align}
\P'(X)=\Exp_{M} [ \P'(X|M) ], \quad \text{ where }\quad 
 \P'(X|M) = \frac{ \P(X | M) \indicator{F_M} }{ \P(F_M | M) }.
 \label{eq:conditional_distribution_def_2}
\end{align}
In both cases, it is straightforward to show that $\P \contig \P'$, and therefore that if the conditional second moment $\Exp_{X \sim \Q} \left(\P'(X)/\Q(X)\right)^{2} $ is bounded, then $ \P \contig \P' \contig Q$. 

In the Gaussian mixture clustering problem, we use the first type of conditional second moment method by conditioning on the typical value of cluster centers $\|V\|_2$. In the sparse PCA problem and the submatrix localization problem, we instead use the second type of conditional second moment method to close the factor of $\sqrt{2}$ between the direct second moment lower bound and the first moment upper bound. A similar method was used in~\cite{PerryWeinBandeira16} to derive the sharp constant of the detection lower bound in the sparse PCA problem. 
Previous work~\cite{arias2013community,verzelen2013sparse} used the conditional second moment method in deriving the sharp detection threshold in community detection problem.

As a high-level motivation for this conditioning, let $\P(X)$ and $\Q(X)$ be the planted and null models $X = M + W$ and $X = W$ respectively, where $M$ is drawn from some prior over symmetric matrices and $W$ is a Wigner matrix.  
Using Fubini's theorem,
\begin{align*}
	\Exp_{X \sim \Q} \,\left(\frac{\P(X)}{\Q(X)}\right)^2
	&= \Exp_{M,M'} \!\left[ \frac{\P( X | M ) \P( X | M' ) }{\Q^2(X) } \right] \\
	&=\Exp_{M,M'} \Exp_{X \sim \Q} \!\left[ \e^{  - \frac{1}{4} \| M \|_F^2 - \frac{1}{4} \| M' \|_F^2 + \frac{1}{2} \langle X, M+M' \rangle} \right] \\
	&=\Exp_{M, M'} \e^{ \frac{1}{2} \inner{M}{M'} }
\end{align*}
where $M$ and $M'$ are drawn independently from the prior, and in the last line we have carried out the the integration with respect to $X \sim \Q$ directly. The reader may refer to the proof of \prettyref{lmm:secondmomentGaussian} for the intervening lines. However, it is possible to decrease the second moment by conditioning on the typical value of $\langle X, M +M' \rangle$ in the planted model.

Specifically, suppose that 
$\P\left[ \langle X, M  \rangle \approx \|M\|_F^2 \mid M\right] = 1+o(1)$ uniformly over $M$. Then, by letting $F_M=\{ \langle X, M  \rangle \approx \|M\|_F^2\}$,
we get that
\begin{align*}
	\Exp_{X \sim \Q} &\, \left(\frac{\P'(X)}{\Q(X)}\right)^2 \\
	& \approx \Exp_{M,M'} \Exp_{X \sim \Q} \!\left[ \exp\!\left(  - \frac{1}{4} \| M \|_F^2 - \frac{1}{4} \| M'\|_F^2 + \frac{1}{2} 
	\langle X, M+M' \rangle \right) \indc{F_M} \indc{F_{M'}} \right] \\
	& \approx  \Exp_{M,M'} \Exp_{X \sim \Q} \!\left[
	 \exp\!\left(  - \frac{1}{4} \| M \|_F^2 - \frac{1}{4} \| M'\|_F^2 + \frac{1}{2} 
	\langle X, M+M'  \rangle \right) \indc{\langle X, M+M' \rangle \approx \|M\|_F^2 +\|M'\|_F^2} 
	\right] \\
	&  \approx \Exp_{M,M'} \!\left[ 
	 \exp\!\left(  \frac{1}{2} \langle M, M'\rangle -  \frac{\|M+M'\|_F^2 }{4} \left( 1  - \frac{ \| M\|_F^2 + \|M'\|_F^2}{  \| M+M' \|_F^2} \right)_+^2   \right) 
	 \right] \\
	 & =  \Exp_{M,M'} \!\left[ 
	 \exp\!\left(  \frac{1}{2} \langle M, M'\rangle  \frac{ \|M\|_F^2 + \|M'\|_F^2}{ \| M+ M' \|_F^2}   \right) 
	 \right] 
\end{align*}
where we used the fact that for $Y \sim \calN(0, \sigma^2)$, 
$
\expect{e^{ Y} \indc{Y\le b} } \le \exp\!\left( \frac{1}{2}  \sigma^2  - \frac{1}{2}  \sigma^2 
(1  - \frac{b}{ \sigma^2}  )_+^2 \right)
$
and $\frac{1}{2} \langle X, M+M' \rangle \sim \calN(0,  \| M+M' \|_F^2/2 )$.
Comparing this conditional second moment with the unconditional one,
there is a correction term $ \left(\|M \|_F^2 + \|M'\|_F^2\right)/\| M+ M' \|_F^2  $, which is most effective if  $M=M'$. It will turn out that in the sparse PCA problem, when the sparsity $\gamma \to 0$ or in the submatrix localization problem, when the number of blocks $k \to \infty$, the second moment is dominated by the event that $M =M'$. Hence,
$$
\Exp_{X \sim \Q} \,\left(\frac{\P'(X)}{\Q(X)}\right)^2 \approx
\exp\!\left(  \frac{1}{4} \|M\|_F^2   \right) 
 \mathbb{P}_{M,M'} \left[ M=M' \right],
$$
as opposed to $\exp\!\left( \frac{1}{2} \|M\|_F^2 \right) \mathbb{P}_{M,M'} \left[ M=M' \right]$ in the absence of conditioning, a factor of two gain in the exponent. 
Notice that there is one slight difference between this informal description and our proofs. Instead of conditioning on 
the typical value of $\langle X, M   \rangle \approx \|M\|_F^2$, we 
condition on the typical value of $\| (X-M)_S \|_2$, where $S=\supp(M)$
and $X_S$ denotes the matrix by setting entries of $X$ outside of $S$ to be zero. The reader may refer to 
\prettyref{sec:conditional_sparsepca} and \prettyref{sec:conditional_submarix} for details.

\subsection{Non-reconstructibility}
\label{sec:non-recon}

Without further embellishment, contiguity is a statement about detection and not reconstruction. It is tempting to believe that whenever contiguity holds---that is, whenever we cannot tell whether a particular sample was generated from the null or planted model---we also cannot infer the planted signal $M_0$ better than chance. 
This is not the case. 
Consider a strange situation in which the null and planted models are identical and noiseless: in $\P$ we observe $X=M_0$ drawn from a prior $\P_0$, and in $\Q$ we observe a random draw $M$ from $\P_0$.  Detection is patently impossible because these models are contiguous, but if we know that $X$ is drawn from the planted model, the reconstruction problem is trivial since we observe the ground truth $M_0$ directly.

However, in the model described by~\prettyref{eq:planted_model} where the noise is additive and Gaussian, we can show that a bounded 
KL divergence implies that reconstruction is impossible as well.  If $M$ is the planted signal, the mean squared error of an estimator $\widehat{M}$ is $\Exp \|M-\widehat{M} \|_F^2$. The following theorem shows that whenever
the KL divergence $D_{\mathrm{KL}} (\P \| \Q)  = o(n)$, 
the estimator $\widehat{M}$ that minimizes the mean squared error tends to the trivial estimator $\hat{M}=0$.
By \prettyref{eq:secondmoment_KL} a bounded second moment~\prettyref{eq:smm-bound} implies a bounded KL divergence;
hence a bounded second moment also implies non-reconstruction.

\begin{theorem} \label{thm:MMSE}
Let $\P(X)$ and $\Q(X)$ be the planted and null models $X = M + W$ and $X = W$ respectively, where $M$ is drawn from some prior such that $\Exp[M]=0$ and $\lim_{n\to\infty} (1/n) \Exp[\|M\|_F^2]$ exists,  and where $W$ is a Gaussian or Wigner matrix.  The MMSE estimator in the planted model is the mean of the posterior distribution: 
$\widehat{E} (X) = \Exp [M | X].$
If the KL divergence $D_{\mathrm{KL}} (\P \| \Q)  = o(n)$, then
%
\begin{align} \label{eq:conditional_mean}
	\liminf_{n\to\infty}  \frac{1}{n} \Exp_{X \sim \P} \|  \widehat{E} (X) \|_F^2 = 0 \, .
\end{align}
It further follows that for any estimator $\hat{M} =\hat{M}(X)$ such that $\Exp_{X} \| \hat{M} \|_F^2 =O(n)$, we have that 
\begin{align}
\liminf_{n\to\infty} \frac{1}{n} \Exp_{M, X} \inner{M}{\hat{M}} =0. \label{eq:overlap_general}
\end{align}
\end{theorem}

\noindent When $M= (\snr/\sqrt n)\, UV^\dagger$, where the rows of $U$ and $V$ are independently and identically distributed according to some priors, 
then the liminf in \prettyref{eq:conditional_mean} and \prettyref{eq:overlap_general} can be replaced by lim.

In cases where we use the conditional second moment method by conditioning on events $F$ that depend only on the signal $M$, \ie, $F=\{M \in F\}$, the conditional distribution $\P'$ given in \prettyref{eq:conditional_distribution_def_1} is still an additive Gaussian model. Hence, we can still apply~\prettyref{thm:MMSE} with $\P'$ and $\Q$,
 concluding that a bounded conditional second moment $\Exp_{X \sim \Q} \left(  \P'(X) / \Q(X) \right)^{2} $ implies non-reconstruction in $\P'$ and hence non-reconstruction in $\P$. 

When we need to condition on events $F_M$ that depend on  both $M$ and $X$, \ie, $F=\{X-M \in F\}$, the conditional distribution $\P'$ given in \prettyref{eq:conditional_distribution_def_2} may no longer be an additive Gaussian model, and hence~\prettyref{thm:MMSE} cannot be directly invoked. Fortunately, we are able to prove that $(1/n) D_{\mathrm{KL}} (\P' \| \Q)$ is an asymptotic upper bound to $(1/n) D_{\mathrm{KL}} (\P \| \Q).$ As a consequence,  by \prettyref{eq:secondmoment_KL} a bounded conditional second moment $\Exp_{X \sim \Q} \left(  \P'(X) / \Q(X) \right)^{2} $implies $D_{\mathrm{KL}} (\P \| \Q) = o(n)$ and hence  non-reconstruction in $\P$ by invoking~\prettyref{thm:MMSE}. 

Let $\|M\|_\ast$ denote the nuclear norm of $M$, which equals to the sum of all the singular values of $M$.
\begin{theorem} \label{thm:conditional_KL}
Let $\P(X)$ and $\Q(X)$ be the planted and null models $X = M + W$ and $X = W$ respectively, where $M$ is drawn from some prior such that $\|M\|_2=O(\sqrt{n})$ and $\|M\|_\ast=O(\sqrt{n})$ uniformly over all $M$, and where $W$ is a Gaussian or Wigner matrix.  
Suppose $\P'$ is given in \prettyref{eq:conditional_distribution_def_2} with $\P(F_M|M) = 1+o(1)$ uniformly over all $M$. 
Then 
$$
  D_{\mathrm{KL}} (\P \| \Q)  \le D_{\mathrm{KL}} (\P' \| \Q)+o(n).
$$
\end{theorem}

\prettyref{thm:conditional_KL} needs the technical assumptions that 
$\|M\|_2=O(\sqrt{n})$ and $\|M\|_\ast=O(\sqrt{n})$. These assumptions are satisfied in the sparse PCA and submatrix localization problems, but not in the Gaussian mixture clustering problem, which has a prior distribution of unbounded support. This last fact does not impact our results, because the conditioning we emply for the clustering lower bound is on the signal $M$ and not the noise, meaning that as above we can invoke~\prettyref{thm:MMSE}. To deal with prior distribution of unbounded support, one could emply a truncation argument.

\section{Proofs}

\subsection{Notation and preliminary lemmas}



We begin by proving  \prettyref{lmm:secondmomentGaussian}. The proof is discussed in~\cite[p.97]{IS03}; we give the full proof here for completeness. 
\begin{proof}[Proof of \prettyref{lmm:secondmomentGaussian}] 
We focus on the Wigner noise case as the proof for Gaussian noise case is identical except for a factor of $2$. As we noted in equation \prettyref{eq:conditional-ratio}, the conditional likelihood ratio is
	\begin{equation}
	\frac{\P(X | M)}{\Q(X)} = \frac{\Q(X-M)}{\Q(X)} 
	= \e^{\frac{1}{2} \inner{X}{M} - \frac{1}{4} \|M \|_F^2} \, .
	\end{equation}
We have $\P(X) = \Exp_{M} \,\P(X|M)$. Reversing the order of the expectations and applying~\prettyref{eq:conditional-ratio} gives
\begin{align*}
\Exp_{X \sim \Q} \,\frac{\P(X)^2}{\Q(X)^2}
&= \Exp_{X \sim \Q} \,\Exp_{M,M'} \,\frac{\P(X|M) \,\P(X|M')}{\Q(X)^2} \\
&= \Exp_{M,M'} \Exp_{X \sim \Q} \,\frac{\P(X|M) \,\P(X|M')}{\Q(X)^2} \\
&= \Exp_{M,M'} \,\e^{-\frac{1}{4} ( \|M \|_F^2 + \|M'\|_F^2)} \,\Exp_{X \sim \Q} \,\e^{\frac{1}{2} \inner{X}{M+M'}} \\
&= \Exp_{M,M'} \,\e^{-\frac{1}{4} ( \|M \|_F^2 + \|M' \|_F^2 - \|M+M'\|_F^2)} \\
&= \Exp_{M,M'} \,\e^{\frac{1}{2} \inner{M}{M'}} \, ,
\end{align*}
where in the second-to-last line we used the moment generating function $\Exp_{X \sim \Q} \,\e^{\inner{X}{A}} = \e^{ \|A \|_F^2}$. This completes the proof.
\end{proof}

In the submatrix localization and Gaussian mixture clustering problems, the underlying low-rank signal matrix $M$  arises from a balanced partition $\sigma_0: [n] \to [k]$. In these cases we can equivalently frame results about reconstruction in terms of how well we can infer this original partition. Given $\sigma_0$ and an estimated partition $\hat\sigma$, we define the \emph{overlap matrix} as a $k\times k$ matrix $\omega(\sigma_0,\hat\sigma)$ which has $s,t$ entry equal to the fraction of integers in $[n]$ assigned by $\sigma_0$ to group $s$ and by $\hat\sigma$ to group $t$, i.e.
\[
	\omega(\sigma_0,\hat\sigma)_{s,t} = \frac{|\sigma_0^{-1}(s) \cap \hat\sigma^{-1}(t)|}{n/k}
\]
and we drop the dependence on $\sigma_0$ and $\hat\sigma$ whenever clear. Our assumption that the partitions are balanced implies that $\omega$ is doubly stochastic.  

We can read off scalar measures of the correlation between two partitions $\sigma$ and $\tau$---i.e. of how well we have reconstructed the planted signal---directly from $\omega$; two will be particularly useful. It is typical in the literature to work with what we will call the \emph{trace overlap} between $\sigma$ and $\tau$, $T(\sigma,\tau) := \max_{\pi}\Tr\pi\omega(\sigma,\tau)$, where the maximum is taken over all permutation matrices. In our problems, however, it is analytically more convenient to use $L(\sigma,\tau):= \norm{\omega(\sigma,\tau)}^2_F$, which we call the \emph{$L_2$ overlap}. Both $L(\cdot)$ and $T(\cdot)$ range from $1$, when the two partitions are uncorrelated and $\omega = \J/k$, to $k$, when they are identical up to a permutation of the group labels, and $\omega$ is the corresponding permutation matrix. The following lemma states that whenever the $L_2$ overlap is bounded above $1$, the trace overlap is as well.

\begin{lemma} \label{lmm:overlap}
	For any doubly stochastic matrix $\omega$, $L(\omega) \le T(\omega)$.
\end{lemma}
\begin{proof}
This lemma is an immediate consequence of Birkhoff's theorem. We simply expand $\omega = \sum_\pi a_\pi \pi$ as a convex combination of permutation matrices $\pi$ and observe that
\[
	L(\omega) 
	= \Tr \omega^\dagger \omega 
	= \sum_{\pi} a_\pi \Tr\pi^\dagger \omega \le \max_{\pi}\Tr\pi^\dagger \omega 
	= T(\omega),
\]
the final inequality following from $\sum_\pi a_\pi = 1$.
\end{proof}
\noindent Thus, to show that an estimator $\widehat \sigma$ achieves trace overlap with the planted partition $\sigma_0$ bounded above one, it is sufficient to show that $L(\widehat \sigma,\sigma_0) > 1$.

 Finally, in the submatrix localization and Gaussian mixture clustering problems, our second moment calculations will reduce to the computation of $\Exp_{\sigma,\tau}\left[\exp(\xi (\|\omega\|_F^2 - 1)) /2\right]$ where the expectation is over uniformly random pairs of balanced partitions $\sigma$ and $\tau$ with overlap matrix $\omega$, and $\xi$ is a parameter corresponding to the signal-to-noise ratio. By~\cite[Lemma 6]{achlioptas-naor}, it is straightforward to prove the following lemma giving a sufficient condition to guarantee that this expectation, and therefore the second moment as a whole, is bounded by a constant. We will state the lemma in notation consistent with \cite{achlioptas-naor} and discuss afterwords.

\begin{lemma}\label{lmm:laplacemethod}
Assume that $\varphi(\omega)$ is an $\R$-valued function of doubly-stochastic $k\times k$ matrices $\omega$ with the properties that $\varphi(\J/k) = 0$ and for some $\delta>0$ and every $k \times k$ doubly stochastic matrices $\omega$,
$$
	H(\omega) + \varphi(\omega) \le H(\J/k) + \varphi(\allones/k) - \delta\left(\|\omega\|_F^2 - 1\right).
$$
Then there exists a constant $C$, dependent on $k$ and $\delta$, such that 
$$
	\Exp_{\sigma,\tau} \left[ \e^{n \varphi(\omega) }  \right] \le C. 
$$ 
\end{lemma}

\begin{proof}
Let $\Omega=(\Omega_{st})$ such that $\Omega_{s,t}=|  \sigma^{-1} (s) \cap \tau^{-1} (t ) | $. Then $\omega= k\Omega /n$. Denote by $\calD$ the set of all $k \times k$ matrices $\Omega=(\Omega_{st})$ of nonnegative integers such that the sum of each row and each column is $n/k$. Notice that for a given $\Omega=(\Omega_{st})$, there are precisely $n! / \prod_{s,t}  \Omega_{s,t}  ! $ pairs of balanced partitions $(\sigma, \tau)$ with overlap matrix  given by $\Omega$. Hence,
\begin{align} \label{eq:second_moment_sub_asymp}
	\Exp_{\sigma,\tau} \left[ \e^{n \varphi(\omega) }  \right] 
	& = \frac{(n/k)!^{2k}}{n!^2} \sum_{ \Omega \in \calD} \frac{n!} {  \prod_{s,t} \Omega_{s,t} !}  \exp \left( n \varphi (\omega ) \right)  \nonumber \\
	& \le  k^{-2n} \frac{e^{2k} n^{k-1}}{ 2 \pi k^k } \sum_{\Omega \in \calD} \frac{n!} {  \prod_{s,t}  \Omega_{s,t} !}  \exp \left(\varphi ( \omega) \right), 
\end{align}
where we have used Stirling's approximation $ \sqrt{2\pi n}(n/\e)^n \le n! \le \e\sqrt n (n/\e)^n$ in the last step.
In view of the assumption and~\cite[Lemma 6]{achlioptas-naor}, there exists a constant $C'>0$ dependent on $\delta $ and $k$ such that 
$$
	\sum_{\Omega \in \calD} \frac{n!} {  \prod_{s,t}  \Omega_{s,t}  !} \e^{ n \varphi (\omega) }
	\le \frac{C'}{n^{k-1}} \left( k^2 \e^{\varphi(\allones/k) } \right)^n = \frac{C'}{n^{k-1}} k^{2n} .
$$
Combing the last displayed equation with \prettyref{eq:second_moment_sub_asymp} yields that
$$
	\Exp_{\sigma,\tau} \left[ \e^{n \varphi(\omega) }  \right]  \le   \frac{C' \e^{2k}  }{ 2\pi   k^k } := C,
$$
which completes the proof.
\end{proof}

Let us unpack the conditions of the above lemma in the context of our problems. For us, $\varphi(\omega) = (\xi/2)\left(\|\omega\|_F^2 - 1\right)$, and since $H(\J/k) = \log k$ and $\|\J/k\|_F^2 = 1$, it is in fact sufficient to study the function
\begin{align}
	\Phi(\omega) = H(\omega) - \log k + \frac{\xi}{2}\left(\|\omega\|_F^2 - 1 \right). \label{eq:phiomega}
\end{align}
In particular, the hypotheses of \prettyref{lmm:laplacemethod} are satisfied provided that $\Phi(\omega) \le -\delta \left(\|\omega\|_F^2 - 1 \right)$ 
for every doubly stochastic matrix $\omega$ and some $\delta > 0$.
This is the case whenever $\xi$ is sufficiently small. On the other hand, in the limit $\xi \to \infty$, $\Phi(\omega)$ is maximized by any permutation matrix, the doubly stochastic matrices with maximal Frobenius norm and minimal entropy. It has been conjectured that, for general $\xi$, the maximizer is a convex combination of $\allones/k$ and a permutation matrix, but this has not been proved.

The precise value of $\xi$ up to which $\Phi(\omega)$ is maximized by $\omega=\allones/k$ is not known. Fortunately, Achlioptas and Naor, in their second moment lower bound on the $k$-colorability threshold for \ER graphs~\cite{achlioptas-naor}, proved an upper bound on $\Phi$ by relaxing to singly stochastic matrices.
\begin{lemma} \label{lmm:A-N}
	\cite[Theorems 7,9]{achlioptas-naor} 
	When $\xi <2\log(k-1)/(k-1)$, $\argmax\, \Phi(\omega) = \allones/k$.
\end{lemma}

One can obtain an even tighter bound in the special case $k=2$. 
Recall that the $2\times 2$ doubly stochastic matrices are a one-parameter family---since the upper left entry uniquely determines the matrix. 
Consequently we can rewrite $\Phi(\omega)$ in \prettyref{eq:phiomega} as
\[
	\Phi(x) =  h(x) + \frac{\xi}{2}(4x^2 - 4x + 1) - \log 2\ . 
\]
%
The second derivative of $\Phi(x)$ is $4 \xi - 1/(x(1-x))$. Hence, if $\xi < 1$,  then $\Phi(x)$ is strictly concave in $x$
and  $\arg \max_{x} \Phi(x) = 1/2$. 
\begin{lemma}\label{lmm:A-N_k_two}
When $k=2$ and $\xi<1$, $\arg \max_{\omega } \Phi(\omega) = \allones/2$. 
\end{lemma}

\subsection{Sparse PCA}

In this section we prove Theorem~\ref{thm:sparse-pca}. First, we show that detection and reconstruction are possible if $\lambda > \lambda^{\up}$ 
using a first moment argument as described in Section~\ref{sec:first}: specifically, if $\lambda > \lambda^{\up}$ then with high probability in $\Q$ there are no $v$ with likelihood as high as that of the ground truth $v_0$, and with high probability in $\P$ all such $v$ are correlated with $v_0$.  
Then, we prove that the second moment of the likelihood ratio is bounded if $\lambda < \lambda^{\low}$; 
as discussed in Section~\ref{sec:contig}, this implies that detection and reconstruction are impossible if $\lambda < \lambda^{\low}$. 

\subsubsection{First moment upper bound for sparse PCA}

Recall that $X = (\lambda/\sqrt{n}) v_0 v_0^\dagger + W$, where $v_0$ is drawn uniformly from $\calV = \{v \in \{\pm \gamma^{-1/2},0\} : |\supp(v)| = \gamma n\}$ and $W$ is a Wigner matrix. For \emph{any} $v \in \calV$, using~\prettyref{eq:conditional-ratio}, the conditional log likelihood ratio is 
\begin{align*}
	\log \frac{\P(X | v)}{\Q(X)} 
	&= \frac{1}{2} \inner{M}{X} - \frac{1}{4} \|M\|_F^2 \\
	&= \frac{1}{2} \frac{\lambda}{\sqrt{n}} \inner{v}{Xv} - \frac{1}{4} \frac{\lambda^2}{n} \|v\|_2^4 \\
	&= \frac{1}{2} \frac{\lambda}{\sqrt{n}} \inner{v}{Xv} - \frac{\lambda^2  n}{4} \, . 
\end{align*}
In the null model, $X = W$ and $\inner{v}{Xv}$ is distributed as $\calN(0,2 n^2)$, giving
\begin{equation} \label{eq:pca-dist-null}
	\log\frac{\P(X | v)}{\Q(X)} \sim \calN \!\left(-\frac{\lambda^2 n}{4},\frac{\lambda^2 n}{2}\right).
\end{equation}
On the other hand, in the planted model
\[
	\inner{v}{Xv} 
	= \inner{v}{Wv} + \frac{\lambda}{\sqrt n} \inner{v}{v_0v_0^\dagger v} 
	= \langle v, Wv \rangle + \frac{\lambda}{\sqrt n} \inner{v}{v_0}^2,
\]
so the conditional log likelihood has a distribution which depends on the inner product between $v$ and the ground truth $v_0$. We write this inner product as $\inner{v}{v_0} = \theta n$, for $\theta \in [-1,1]$ so that
\begin{align*}
	\log \frac{\P(X \mid v)}{\Q(X)} \sim \calN \!\left(\frac{(2\theta^2 - 1)\lambda^2 n}{4}, \frac{\lambda^2 n}{2}\right)
\end{align*}
Whenever $\theta = 0$ so that $v$ is uncorrelated with $v_0$, this distribution is identical to that in the null model.

To show that detection is possible above~\prettyref{eq:sparse-pca-up}, notice that in the planted model the maximum likelihood estimate $\hat{v} = \argmax_{v \in \calV} \P(X|v)$ has conditional log likelihood at least as large as the ground truth $v_0$.  By standard Gaussian tail bounds and setting $\theta = \pm 1$ above,
\begin{equation*}
	\P\left[\log\frac{\P(X | v_0)}{\Q(X)} \le \frac{\lambda^2  n}{4} - O(\sqrt{n\log n}) \right] \le n^{-\Omega(1)}.
\end{equation*}
In the null model, Gaussian tail bounds give us
\[
	\Q \left[ \log \frac{\P(X | v)}{\Q(X)} \ge \frac{\lambda^2  n}{4} - O(\sqrt{n \log n}) \right] \le \exp\!\left( - \frac{\lambda^2  n}{4} + O(\sqrt{n\log n}) \right) \, . 
\]
Taking the union bound over all $2^{\gamma n} {n \choose \gamma n}$ possible $v$ and invoking Stirling's formula ${n\choose \gamma n} \le \,\e^{n h(\gamma)}$, we have  
\begin{equation} \label{eq:sparse-pca-union}
	\Q \left[ \max_{v \in \calV} \log \frac{\P(X | v)}{\Q(X)} \ge \frac{\lambda^2 n}{4} - O(\sqrt{n \log n}) \right] 
	\le \exp^{ n \left( -\frac{\lambda^2 n}{4} + h(\gamma)  + \gamma \log 2 \right) + O(\sqrt{n\log n} ) } \, .
\end{equation}
When this expression is $\e^{-\Omega(n)}$, we can with high probability distinguish the null and planted models with the generalized likelihood test. This occurs when $\lambda < \lambda^{\up}$ where $\lambda^{\up}$ is defined in~\prettyref{eq:sparse-pca-up},  \ie, 
\begin{equation}
\label{eq:sparse-pca-first}
	\frac{\lambda^2}{4} > h(\gamma) + \gamma \log 2 \, .
\end{equation}

Next, we prove that reconstruction is also possible above this bound.  Suppose that $X$ is generated from the planted model, and that $\inner{\hat{v}}{v_0} = \theta n$.  We bound the probability that $\hat{v}$ has conditional log likelihood as large as $v_0$ with a union bound over $\hat{v}$ correlated with $v_0$, but we bound the number of such vectors generously as $2^{\gamma n}{n\choose \gamma n}$ once again. Combining this with the Gaussian tail bound and invoking Stirling gives
\begin{align*}
	%
	& \P \!\left[ \max_{ v \in \calV:  \langle v,v_0 \rangle \le \theta  n} \log\frac{\P(X | v)}{\Q(X)}
	> \frac{\lambda^2  n}{4} + O(\sqrt{n\log n}) \right] \\
	&\qquad\qquad\qquad \le \exp\!\left[ n \left( \frac{(2\theta^2 - 1)\lambda^2 }{4} + h(\gamma) + \gamma \log 2 \right) 
	+ O(\sqrt{n \log n}) \right] \, . 
\end{align*} 
Since the coefficient of $n$ in the exponent is an analytic function of $\theta$, whenever~\eqref{eq:sparse-pca-first} holds, the RHS of the last displayed equation is exponentially small unless $\theta^2 > \eps$ for some constant $\eps > 0$.  Therefore, with high probability the MLE estimator has overlap $(\inner{\hat{v}}{v_0}/n)^2 = \theta^2 > \eps > 0$ with the ground truth.  Moreover, if $\hat{M} = (\lambda/\sqrt{n} ) \hat{v} \hat{v}^\dagger$, then $\|\hat{M}\|_F^2 = (\lambda^2/n) \|\hat{v}\|_2^4 = \lambda^2 n$;
moreover, with high probability $\inner{\hat{M}}{M} = (\lambda^2 / n) \inner{\hat{v}}{v_0}^2 > \eps \lambda^2 n$.
Hence, the estimator $\hat{M}$ reconstructs the signal matrix better than chance.

\subsubsection{Second moment lower bound for sparse PCA}

In this subsection, we use the second moment method to prove a lower bound on the detectability transition for sparse PCA. 
 We assume $\lambda<1$ throughout the proof; 
the boundary case $\lambda=1$ is not addressed. 
In the planted model, the signal matrix is $M = (\lambda / \sqrt{n}) vv^\dagger$, and in the null model it is zero.  In both cases, we have $X=M+W$ where $W$ is a Wigner noise matrix. 
Applying~\prettyref{lmm:secondmomentGaussian}, the second moment (where $X$ is drawn from the null model) is 
\begin{align} \label{eq:second}
	\Exp_{X \sim \Q}  \left( \frac{\P(X)}{\Q(X)} \right)^{\!2\,}  = \Exp_{v,w} \, \e^{\frac{\lambda^2}{2n} \inner{v}{w}^2}  \, ,
\end{align}
where $v$ and $w$ are drawn independently from the prior.

Denote the overlap $z = |\supp(v) \cap \supp(w)|$ and let $t$ be the difference between the number of indices in that intersection where $v$ and $w$ agree and the number of indices where they disagree. In that case,
\[
\inner{v}{w} = \frac{t}{\gamma}\ .
\]
Then, $z$ follows an hypergeometric distribution with parameter $(n,\gamma n,\gamma n)$ and given $z$, $t$ is distributed as a sum of $z$ independent Rademacher random variables. If we write 
\begin{equation}
\label{eq:lambda}
\eta=\lambda/\gamma \, , 
\end{equation}
then
\[
	\Exp_{X \sim \Q} \left( \frac{\P(X)}{\Q(X)} \right)^{\!2\,} 
	= \Exp_t  \left[ \e^{ \frac{ \eta^2 t^2}{2n} } \right] \, .
\]
We know from~\cite[p.173]{aldous85} that $z$ has the same distribution as the random variable $\mathbb{E}[\overline{z}|\mathcal{B}_n ]$ where $\overline{z}$ is a Binomial random variable with parameters $(\gamma n,\gamma)$ and $\mathcal{B}_n$ some suitable $\sigma$-algebra. Given $\overline{z}$, let $\overline{t}$ be distributed as a sum of $\overline{z}$ independent Rademacher variables. 
\begin{lemma}
Let $(\epsilon_i)_{i\geq 1}$ denote a sequence of independent Rademacher random variables. For any $a>0$,  let $g$ be the piecewise linear function on $[1,\infty)$ such that,
for any positive integer $z$, $g(z)= \Exp_{\epsilon}\exp[a(\sum_{i=1}^z \epsilon_i)^2]$. Then, the function $g$ is convex.
\end{lemma}

\begin{proof}
Since $g$ is convex and continuously differentiable on each interval of the form $(z,z+1)$ where $z$ is an integer, we only have to prove that its left derivative is less or equal to its right derivative at $z+1$. This is equivalent to showing that $g(z+2)+g(z)\geq 2 g(z+1)$. Define $u=\sum_{i=1}^z \epsilon_i$. Conditioning with respect to $u$, we obtain $g(z+1)= \Exp_{u}[\e^{au^2+a} \cosh 2au]$ and $g(z+2)= \Exp_{u}[\e^{au^2}[0.5 \e^{4a} \cosh 4au + 0.5]$. Hence, it suffices to prove that, for any $u$, 
\[
  \e^a \cosh 2au \leq \frac{1}{4} \e^{4a} \cosh 4au + \frac{3}{4} \, .
\]
For $x\geq 0$, let $h(x) = \e^{ax} \cosh(2au\sqrt{x})$. Since $h$ is a product of two increasing convex functions, $h$ is also convex. This implies the above inequality and concludes the proof. 
\end{proof}

By Jensen's inequality, it follows that 
\begin{align} \label{eq:sparse_pca_second_moment_1}
	\Exp_{X \sim \Q} \left( \frac{\P(X)}{\Q(X)} \right)^{\!2\,} 
	\leq  \Exp_{\overline{t}}  \left[ \e^{ \frac{ \eta^2 \overline{t}^2}{2n} } \right] \, .
\end{align}
To simplify the notation, we simply write $z$ and $t$ instead of $\overline{z}$ and $\overline{t}$ in the sequel. Note that 
\begin{align}
	\Exp_t \!\left[  \e^{ \frac{ \eta^2 t^2}{2n} } \mid z\right] 
	&= \sum_{u=0}^{z} \e^{\frac{ \eta^2 u^2}{2n} } \left( \P_{t}\left[|t|\geq u \mid z \right]- \P_{t}\left[|t|\geq u+1  \mid z\right] \right) \nonumber \\
	&\le 1 + \sum_{u=1}^z \frac{\eta^2 u}{n} \e^{\frac{ \eta^2 u^2}{2n} } \P_{t}\left[|t|\geq u \mid z \right] \label{eq:upper_conditional} \\
	&\le 1 + 2\sum_{u=1}^z \frac{\eta^2 u}{n} \e^{\frac{  u^2}{2}\left(\frac{\eta^2}{n} - \frac{1}{z} \right) }\ ,\label{eq:upper_conditional2}
	\end{align}
where we applied Hoeffding's inequality: $\P_{t}(|t|\geq u \mid z) \le 2 \e^{-u^2/(2z) }$ in the last line. 

We consider three subcases depending on the value of $z$. 

\noindent {\bf Case 1}: If $z \leq (n/\eta^2) \lambda$,
\begin{eqnarray} \nonumber
 \Exp_t \!\left[  \e^{ \frac{ \eta^2 t^2}{2n} } \mid z\right]
 &\leq & 1 + 2\sum_{u=1}^z \frac{\eta^2 u}{n} \e^{\frac{u^2}{2}\left(\frac{\eta^2}{n}- \frac{1}{z} \right) } \leq 1 + O(1)\int_0^{\infty}\frac{\eta^2 u}{n} \e^{\frac{  -u^2}{2}\left( \frac{1}{z}-\frac{\eta^2}{n} \right)}du\\
 &\leq & O(1)\left[1 + \frac{\eta^2}{n/z- \eta^2}\right]\leq O\left(\frac{1}{1-\lambda}\right) \label{eq:upper_small_z}\ ,		
\end{eqnarray}
where we used $z \le \frac{n\lambda}{\eta^2}$ in the last inequality. \\

\noindent {\bf Case 2}: If $(n/\eta^2) \lambda < z \le n/\eta^2$, we have
\begin{eqnarray}
  \Exp_z\left[\Exp_t \!\left[  \e^{ \frac{ \eta^2 t^2}{2n} } \mid z\right]\indicator{z\in (\tfrac{n}{\eta^2}\lambda , \tfrac{n}{\eta^2}]}\right]&\leq &\Exp_z\left[\left(1 + 2\sum_{u=1}^z \frac{\eta^2 u}{n} \e^{\frac{u^2}{2}\left(\frac{\eta^2}{n}- \frac{1}{z} \right) }\right) \indicator{z\in (\tfrac{n}{\eta^2}\lambda , \tfrac{n}{\eta^2}]}\right] \nonumber \\
  & \leq & \Exp_z\left[ \left( 1+  \frac{\eta^2 (z+1)^2}{n} \right) \indicator{z\in (\tfrac{n}{\eta^2} \lambda, \tfrac{n}{\eta^2}] }\right] \nonumber \\
  &\leq  &  O( 1 + n \gamma )\P_z\left[z\geq \tfrac{n}{\eta^2}\lambda \right]\nonumber, 
\end{eqnarray}
where in the second inequality we bound $\exp\!\left( u^2/2\, \left(\eta^2/n - 1/z \right) \right) \le 1$ and in the third inequality we use $z\leq n/\eta^2$ and $z \le n\gamma$. 
Since $\lambda / \eta^2 = \gamma^2 / \lambda > \gamma^2$ 
and $z \sim \Binom(\gamma n,\gamma)$, Bernstein's inequality implies that $\P_z\left[z \geq (n / \eta^2) \lambda \right]$ is exponentially small. Therefore,  
\begin{equation} \label{eq:upper_medium_z}
  \Exp_z\!\left[\Exp_t \!\left[ \e^{ \frac{ \eta^2 t^2}{2n} }|z\right]\indicator{z\in (\tfrac{n}{\eta^2}\lambda, \tfrac{n}{\eta^2}]}\right] = o(1)\ . 
\end{equation}

\noindent {\bf Case 3}: Finally, if $n/\eta^2 < z \le \gamma n$, then $\exp\!\left( u^2/2\, \left(\eta^2/n - 1/z \right) \right)$ is exponentially large in $u$. Thus to show that the second moment is finite, we need to use the fact that $\P_z[ z \geq a]$  for  $a>n/\eta^2$ is exponentially small. When $z> n/\eta^{2}$, the bound in\eqref{eq:upper_conditional2} implies that
\[
	\Exp_t \!\left[  \e^{ \frac{ \eta^2 t^2}{2n} } \mid z\right]\leq 1 + 2 \e^{\frac{z^2\eta^2}{2n} - \frac{z}{2} } \sum_{u=1}^z \frac{\eta^2 u}{n} \leq O(n)   \exp\!\left[\frac{z^2\eta^2}{2n} - \frac{z}{2} \right]\ , 
\]
where we applied $z\leq \gamma n$ and the second inequality holds due to $\lambda<1$. 
For any  $0<p_0< p_1\leq 1$, define 
\begin{equation} \label{eq: hp-def}
	h_{p_0}(p_1)=p_1\log(\tfrac{p_1}{p_0})+ (1-p_1) \log\left( \frac{1-p_1}{1-p_0} \right) \, ,
\end{equation}
the KL-divergence between Bernoulli random variables with parameters $p_0$ and $p_1$ respectively. Chernoff's inequality implies that for any $a\geq \gamma^2 n$, 
\[
 \P_z[ z \geq a]\leq \exp \left [- \gamma n h_{\gamma}\left(\frac{a}{\gamma n}\right)\right ] \ . 
\]
As a consequence, 
\begin{eqnarray*}
  \Exp_z \!\left[\Exp_t \!\left[  \e^{ \frac{ \eta^2 t^2}{2n} } \mid z \right] \indicator{z\in (\tfrac{n}{\eta^2},\gamma n ]} \right]& \leq& O(n)\sum_{z=n/\eta^2}^{\gamma n} \exp\!\left[\frac{z^2\eta^2}{2n} - \frac{z}{2} - \gamma n h_{\gamma}\left(\frac{z}{\gamma n}\right)\right]\\
  &\leq & O(n^2) \exp\!\left[ n \sup_{\zeta\in [1/\eta^2,\gamma]}  \psi(\zeta) \right] \, , 
\end{eqnarray*}
where 
\begin{equation}\label{eq:psi}
\psi(\zeta) 
= \frac{(\eta^2 \zeta-1)\zeta}{2}- \gamma h_{\gamma}(\zeta/\gamma)
 = \frac{(\eta^2 \zeta-1)\zeta}{2} - \zeta \log\left(\frac{\zeta}{\gamma }\right) - (\gamma - \zeta)  \log\left(\frac{\gamma -\zeta}{1-\gamma }\right)+  \gamma \log(\gamma)\ , 
\end{equation}
Then,  we need to show that 
\begin{equation}
\label{eq:cond}
\psi(\zeta) < 0 
\quad \text{for all} \quad 
\zeta \in [\gamma^2/\lambda^2, \gamma] \, ,
\end{equation}

Note that this property is trivial when $\gamma/\lambda^2>1$ so that we may restrict our attention to $\gamma/\lambda^2\leq 1$. Since $\lambda <1$, it follows that $\psi(\gamma^2/\lambda^2)= - \gamma h_\gamma (\gamma/\lambda^2)<0$.
For $\zeta=\gamma$, $\psi(\zeta)= (\lambda^2-\gamma)/2 -\gamma \log(1/\gamma)\ ,$ which is negative for 
\begin{equation}\label{eq:cond2}
 \lambda< \sqrt{2 \gamma(- \log \gamma + 1/2)}\ .
 \end{equation}
Moreover, the second derivative of $\psi$ has at most two roots.  It follows that $\psi(\zeta)$ can have at most three local minima or maxima occurring at the roots of 
\[
\psi' = \eta^2 \zeta - \frac{1}{2} - \log \left(\frac{\zeta (1-\gamma)}{\gamma(\gamma-\zeta)}\right) \, . 
\]
Figure~\ref{fig:sparse-pca-plot} shows $\lambda$ as a function of $\gamma$, illustrating in particular that it rapidly approaches the spectral transition $\lambda= 1$ as $\gamma$ approaches $1$.  

\begin{figure}
\begin{center}
\includegraphics[width=3in]{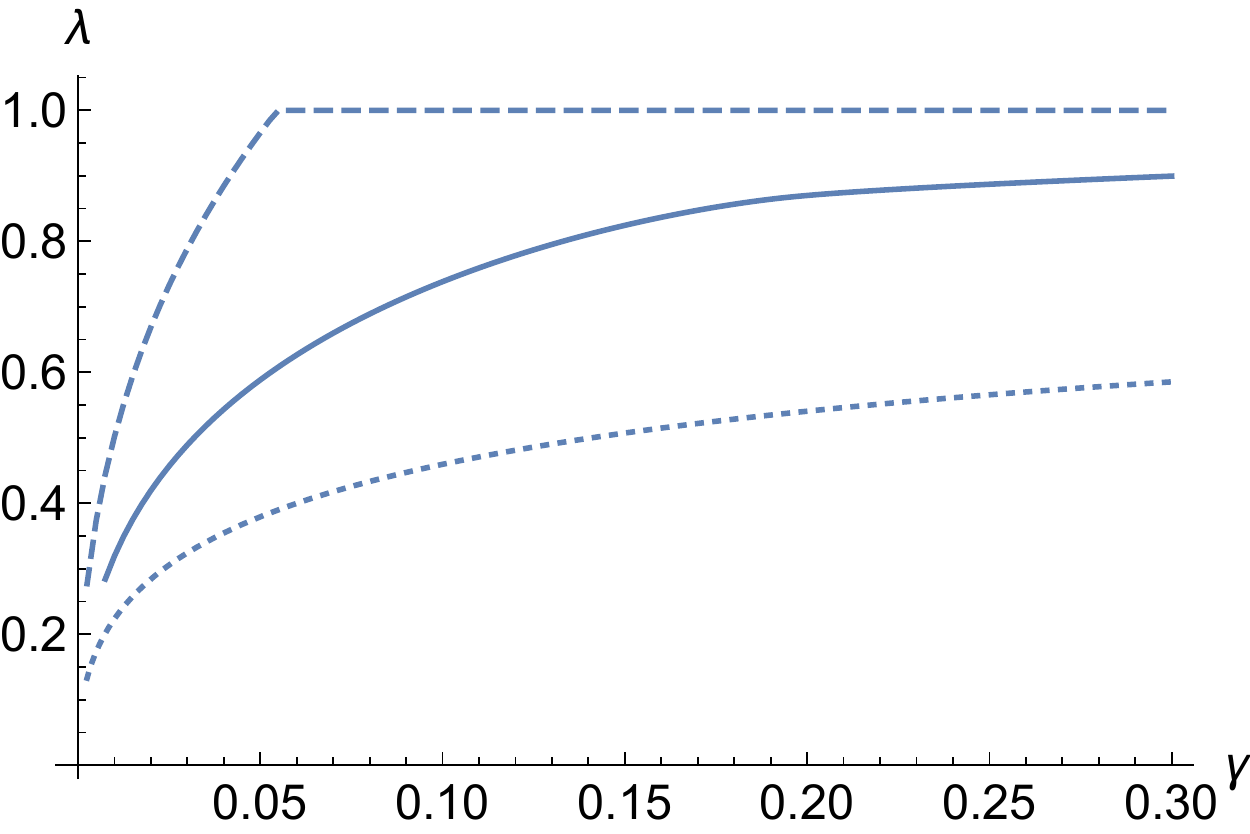}
\end{center}
\caption{Upper and lower bounds on the information-theoretic detectability threshold in sparse PCA.  Our upper bound (dashed) shows that the detectability threshold is strictly below the spectral threshold $\lambda=\gamma \eta=1$ for $\gamma < 0.054$.  The solid line shows the lower bound given by~\eqref{eq:cond}, and the dotted line shows the simpler lower bound given by~\eqref{eq:lambert}.
Notice that  the solid line is not touching the dashed line even in the limit $\gamma \to 0$; there is a gap of $\sqrt{2}$. }
\label{fig:sparse-pca-plot}
\end{figure}

We now derive an analytic bound of $\psi$. 
For $\zeta\in (\gamma^2/\lambda^2, \gamma]$, 
\begin{align*}
	\psi(\zeta) 
	&\le \zeta \frac{\eta^2 \gamma - 1}{2} - \zeta \log\left( \frac{\zeta}{\gamma^2} \right) + \gamma (1- \zeta/\gamma)\log\left( \frac{1-\gamma}{ 1-\zeta/\gamma} \right) \\
	&\le \zeta \frac{\eta^2 \gamma - 1}{2} - \zeta \log\left( \frac{\zeta}{\gamma^2} \right) +\zeta - \gamma^2 \\
 	&\le \zeta\left( \frac{\lambda^2/\gamma + 1}{2} + 2 \log \lambda \right) \, ,
\end{align*}
where we used $\log(1+x)\leq x$ in the second line and $\zeta\geq \gamma^2/\lambda^2 \ge \gamma^2$ in the third line.

Therefore, $\Exp_z \!\left[\Exp_t \!\left[ \exp\left(\eta^2t^2/2n\right) \mid z\right]\indicator{z\in (n/\eta^2,\gamma n ]}\right]$, and hence the second moment, are bounded as long as
\begin{align} \label{eq:secondmoment_condition_1}
	\log  \lambda^2 + \frac{\lambda^2/\gamma + 1}{2} < 0 \, . 
\end{align}
If, for any $y>0$, we use $\calW(y)$ to denote the  root $x$ of $x \e^x = y$, \prettyref{eq:secondmoment_condition_1} holds whenever
\begin{equation} \label{eq:lambert}
	\lambda <  \sqrt{ 2 \gamma \calW\!\left( \frac{1}{2 \sqrt{\e} \,\gamma} \right)} \, . 
\end{equation}
As $y \to \infty$, we have $\calW(y) \sim \log y$.  Thus as $\gamma \to 0$, this lower bound approaches
\begin{equation} \label{eq:lambert-simple}
	\lambda^{\low} \sim \sqrt{ 2 \gamma \log \frac{1}{2 \sqrt{\e} \,\gamma} } \, .
\end{equation}
\bigskip 

Let us now turn to regime $\gamma >0.6$, where we will show the second moment is bounded whenever $\lambda<1$. In view of \eqref{eq:upper_small_z} and \eqref{eq:upper_medium_z}, it suffices to prove that
\begin{eqnarray*}
  \Exp_z \!\left[\Exp_t \!\left[ \e^{ \frac{ \eta^2 t^2}{2n} } \mid z\right]\indicator{z\in (\gamma^2 n ,\gamma n ]}\right]= O(1)\ ,
\end{eqnarray*}
whenever $\lambda<1$. From \eqref{eq:upper_conditional}, we have
 \begin{eqnarray*}
	\lefteqn{\Exp_z \!\left[ \Exp_t \!\left[ \e^{ \frac{ \eta^2 t^2}{2n} } \mid z\right]\indicator{z\in (\gamma^2 n ,\gamma n ]}\right]}&& \\
	&\leq  & \Exp_z \!\left[1 + \sum_{u=1}^z \frac{\eta^2 u}{n} \e^{\frac{ \eta^2 u^2}{2n} } \P_{t}\left[|t|\geq u \mid z\right] \indicator{z\in (\gamma^2 n ,\gamma n  } \right]  \\
	&\stackrel{(a)}{\leq } & 1 + \sum_{z=\gamma^2 n }^{\gamma n}\sum_{u=1}^z  \frac{\eta^2 u}{n} \e^{ \frac{ \eta^2 u^2}{2n}} \P_{t}\left[ |t|\geq u \mid z\right]  \binom{\gamma n}{z}\gamma^z (1-\gamma)^{\gamma n-z}   \\
	&\stackrel{(b)}{\leq } & 1+ O(1)\sum_{z= \gamma^2 n}^{\gamma n }\sum_{u=1}^z \frac{\eta^2 u}{n\sqrt{ \left( z (1- z/\gamma n ) \right) \vee 1}}  \e^{-\frac{(1-\lambda^2) u^2}{2n\gamma^2}} \exp\!\left[\frac{u^2}{2n\gamma^2} + z h\!\left( \frac{u}{2z}+\frac{1}{2} \right) - z\log(2) - \gamma n h_{\gamma}\!\left( \frac{z}{\gamma n} \right) \right]   \ ,
 \end{eqnarray*}
where we used in $(a)$ that $z$ follows a Binomial distribution. In step (b), $h_\gamma$ is defined as in \prettyref{eq: hp-def}, and we have used the facts that 
$$ 
	\P_{t}\left[|t|\geq u \mid z\right]  = 2 \Pr\!\left[\Binom(z,1/2) \ge \frac{u+z}{2} \right] 
	\le 2 \exp\!\left( -z h_{1/2}\!\left(\frac{u+z}{2z} \right) \right) 
$$ 
and 
$$
	\binom{\gamma n}{z} \le \left( \frac{1}{ \sqrt{2 \pi z (1-z/(\gamma n) ) } } \wedge 1 \right) \exp\!\left( \gamma n h\!\left( \frac{z}{\gamma n} \right) \right)
$$ 
for $z \le \gamma n$. 

Expanding as Taylor series, we have 
\[
	h(u)\leq \log(2) - 2(u-1/2)^2 - \frac{4}{3}(u-1/2)^4 
	\quad \text{and} \quad 
	h_{\gamma }(u)\geq \frac{(u-\gamma)^2}{2\gamma (1-\gamma )} \, . 
\]
This yields
\begin{eqnarray*}
	& &  \lefteqn{\Exp_z\!\left[\Exp_t \!\left[\e^{ \frac{ \eta^2 t^2}{2n} }| z\right]\indicator{z\in (\gamma^2 n ,\gamma n ]}\right]} \\
	&\leq & 1 +O(1)\sum_{z= \gamma^2 n}^{\gamma n }\sum_{u=1}^z \frac{\eta^2 u}{n\sqrt{ \left(z (1-z/\gamma n) \right)\vee 1}} \e^{-\frac{(1-\lambda^2) u^2}{2n\gamma^2}} \exp\!\left[n\underbrace{\left(\frac{u^2}{2(\gamma n)^2} - \frac{u^2}{2zn}- \frac{u^4}{12z^3n} - \frac{(z/(\gamma n) - \gamma)^2}{2(1-\gamma)}\right)}\right] \\
	  && \hspace{11cm}=(I) 
\end{eqnarray*}

We focus on the term (I), which is maximized with respect to $u$ by choosing 
\[ 
u^2 = \frac{1}{2} \times 12 z^3 n \left( \frac{1}{2 \gamma^2n^2 } - \frac{1}{2 z n} \right) \, . 
\] 
Hence,
 \begin{eqnarray*}
 (I) &\leq& \frac{3}{4}\frac{z^3}{\gamma^2n^3}\left(\frac{1}{\gamma}-\frac{\gamma n}{z}\right)^{2} -  \frac{(z/(\gamma n) - \gamma)^2}{2(1-\gamma)}\\
 &\leq &\left[\frac{3z}{4\gamma^2n} -\frac{1}{2(1-\gamma)}\right]\left(\frac{z}{\gamma n} - \gamma\right)^2\\
 &\leq &\left[\frac{3}{4\gamma} -\frac{1}{2(1-\gamma)}\right]\left(\frac{z}{\gamma n} - \gamma\right)^2 \ ,       
       \end{eqnarray*}
since $z\leq \gamma n$. The above expression is negative as soon as $\gamma>0.6$.
For such a choice of $\gamma$, there exist two positive constants $c_1$ and $c_2$ such that 
\[
\Exp_z\left[\Exp_t \!\left[ \e^{ \frac{ \eta^2 t^2}{2n} }|z\right]\indicator{z\in (\gamma^2 n ,\gamma n ]}\right] \leq 1 +O(1) \sum_{z= \gamma^2 n}^{\gamma n} \sum_{ u= 1}^{z} \frac{u}{n\sqrt{ \left( z ( 1- z/ \gamma n) \right)\vee 1}} \e^{-c_1 \frac{u^2}{n} }e^{- c_2\frac{(z-\gamma^2 n)^2}{n}} = O(1) \ ,
\]
where the last equality holds because $\sum_{u=1}^z (u/n) \e^{-c_1 u^2/n} = O(n)$.
This completes the proof in the regime $\gamma > 0.6$. \\

The conclusion for reconstruction follows by applying \prettyref{eq:secondmoment_KL}
and~\prettyref{thm:MMSE}.  Furthermore, for any estimator $\hat{v} \in \calV$ of $v$, applying~\prettyref{thm:MMSE} to the estimator $\hat{M} = ( \lambda/\sqrt{n} ) \hat{v} \hat{v}^\dagger$ gives
\begin{align*}
 \frac{1}{n^2} \Exp \inner{v}{\hat{v}}^2 =  \frac{1}{n \lambda^2 } \Exp \inner{M}{\hat{M} }  =  o(1) \, . 
\end{align*}
Thus we can neither reconstruct the signal matrix nor the sparse vector $v$.

\subsubsection{Conditional second moment lower bound for sparse PCA} \label{sec:conditional_sparsepca}

Bounding the second moment is too rough to pinpoint the exact information-theoretical threshold for small $\gamma$. 
Indeed, there is asymptotically an $\sqrt{2}$ gap between the $\lambda^{\low}$ and $\lambda^{\up}$. 
To recover the exact constant, we shall bound  some conditional second moment. 
Throughout this proof, we consider $\lambda$ such  that 
\begin{equation} \label{eq:cond_lambda_small_gamma}
	\lambda^2 < 4 \gamma\left[\log\left(\frac{1}{\gamma}\right)-2.1\sqrt{2\log\left(\frac{1}{\gamma}\right)} - \frac{3}{2}\log\left(\frac{3\e}{1-\gamma}\right)\right]\leq 4\gamma \log\left(\frac{1}{\gamma}\right)\ . 
\end{equation}
We shall restrict our attention to $\gamma$ small enough that $\log (1/\gamma)>41+\log(81)$, which is equivalent to
\begin{align}
9 \sqrt{\gamma} < \sqrt{ 2 \gamma \calW\!\left( \frac{1}{2 \sqrt{\e} \,\gamma} \right)} .\label{eq:gamma_regime2}
\end{align}

 For any set $a \subset [n]$ and any square matrix $X$, let $X_a$ refer to the matrix whose entries outside $a\times a$ have been set to zero. 
 Given a vector $v$, we define the event $\Gamma_{v}$ by
 \begin{equation}\label{eq:definition:gamma_v}
 \Gamma_{v}:= \left\{ \|(X-\frac{\lambda}{\sqrt{n}} vv^{\dag})_{\supp(v)}\|_2\leq 2.1 \sqrt{n\gamma}\right\}\ .
 \end{equation}
 Under $\mathbb{P}(\cdot|v)$, the $\supp(v)\times \supp(v)$ submatrix of $X-\frac{\lambda}{\sqrt{n}} vv^{\dag}$ is distributed as a Wigner Matrix. By~\cite[Theorem 5.1]{BaiSilverstein10}, its spectral norm rescaled by $\sqrt{\gamma n}$ converges almost surely to $2$ when $n\to \infty$, and so $\mathbb{P}(\Gamma_v|v)=1+o(1)$ uniformly with respect to $v$. The event $\Gamma_v$ means that the noise added to the nonzero entries of $vv^\dagger$ is not uncharacteristiaclly large.

Next, we condition on the high probability events $\Gamma_v$ and define the conditional probability distribution
 \begin{equation}\label{eq:def_P'_sparse_PCA}
 	\mathbb{P}'(X)= \mathbb{E}_{v}\left[\mathbb{P}(X|v)\frac{\indicator{\Gamma_v}}{\mathbb{P}(\Gamma_v|v)} \right]\ .
 \end{equation}
 The conditional second moment decomposes as 
 \begin{eqnarray*} 
 	\Exp_{X\sim \mathbb{Q}}\left[\frac{\mathbb{P}'^2(X)}{\mathbb{Q}^2(X)}\right]&=& \Exp_{v,w}\left[\Exp_{X\sim \mathbb{Q}}  \left[ \frac{\mathbb{P}(X|v)\mathbb{P}(X|w)}{\mathbb{Q}^2(X)} \frac{\indicator{\Gamma_{v}}\indicator{\Gamma_{w}}}{\mathbb{P}(\Gamma_{v}|v)\mathbb{P}(\Gamma_{w}|w)}\right]\right] \\
 	&\leq  &(1+o(1))\mathbb{E}_{v,w}\left[\underbrace{\Exp_{X\sim \mathbb{Q}}\left[e^{\frac{\lambda}{2\sqrt{n}}<vv^{\dag}+ ww^{\dag},X>- \frac{\lambda^2n}{2}}\indicator{\Gamma_{v}}\right]}\right]\ , \\ 
 	&& \hspace{5cm}=(I)
 \end{eqnarray*}   
 where we used that $\mathbb{P}(\Gamma_v|v)=1+o(1)$. Introduce the notation $z=|\supp(v)\cap \supp(w)|$, $t_+= |\supp(v+w)|$, $t_{-}=|\supp(v-w)|$ and $t=t_+-t_{-}$, so that that  $z=t_++t_-$. Finally, we define the function
 \[
  g(t_+,t_-):= \frac{\lambda\left(t_+^2+t_-^2\right)}{\gamma^2\sqrt{n}}+  2.1 \sqrt{\frac{n}{\gamma}}(t_++t_-)\ . 
 \]
Under $\Gamma_v$, one has $\langle vv^{\dag},X_{\supp(v+w)}+ X_{\supp(v-w)}\rangle \leq g(t_+,t_-)$. To bound $(I)$, we first integrate with respect to entries of $X$ outside $\supp(v+w)\times \supp(v+w)\bigcup \supp(v-w)\times \supp(v-w)$. 
\begin{eqnarray*}
	(I)&=  &   \Exp_{X\sim \mathbb{Q}}\left[ \exp \left( \frac{\lambda}{\sqrt{n}}\langle vv^{\dag},X_{\supp(v+w)}+X_{\supp(v-w)} \rangle - \frac{\lambda^2 z^2}{2\gamma^2n} \right) \indicator{\Gamma_{v}}\right] \\
 	&\leq &  \Exp_{X\sim \mathbb{Q}}\left[ \exp \left( \frac{\lambda}{\sqrt{n}}\langle vv^{\dag},X_{\supp(v+w)}+X_{\supp(v-w)} \rangle- \frac{\lambda^2z^2}{2\gamma^2n} \right) \indicator{ \langle vv^{\dag},X_{\supp(v+w)}+X_{\supp(v-w)}\rangle \leq g(t_+,t_-)}\right] \\
 	& \stackrel{(a)}{\leq}  & \exp\!\left[\frac{\lambda^2t^2}{2\gamma^2n}-\frac{1}{2}\left(\frac{\lambda\sqrt{2(t_+^2 + t_-^2)}}{\gamma \sqrt{n}}-\frac{ g(t_+,t_-)\gamma}{\sqrt{2(t_+^2+ t_-^2)}} \right)_+^2\right] \\
 	&\leq & \exp\!\left[\frac{\eta^2 t^2}{2 n}-\frac{1}{2}\left(\frac{\eta \sqrt{z^2+ t^2}}{2\sqrt{n}}-2.1 \sqrt{n\gamma} \right)_+^2\right]\ ,
 \end{eqnarray*} 
 where we used  
 \[
 \mathbb{E}_{Y\sim \mathcal{N}(0,1)}[e^{aY}\indicator{Y\leq b}] \leq \e^{a^2/2-(a-b)_+^2/2} 
 \]
 and the fact that $ \langle vv^{\dag},X_{\supp(v+w)}+X_{\supp(v-w)}\rangle \sim \calN(0, 2(t_+^2+t_-^2)/\gamma^2)$ in $(a)$ and $\eta= \lambda/\gamma$ in the last line. When $v$ and $w$ are sampled independently according to the sparse PCA model, then $z$ follows a hypergeometric distribution with parameters $(n,\gamma n,\gamma n)$. Conditioned on $z$, $t$ is distributed as a sum of $z$ independent Rademacher random variables. Returning to the second moment, we arrive at 
 \begin{equation}\label{eq:upper_second_moment}
  \Exp_{X\sim \mathbb{Q}}\left[\frac{\mathbb{P}'^2(X)}{\mathbb{Q}^2(X)}\right] 
  \leq (1+o(1))\,\Exp_{z,t}\left[\exp\!\left[\frac{\eta^2 t^2}{2n} -\frac{1}{2}\left(\frac{\eta\sqrt{z^2+t^2} }{2 \sqrt{n}}- 2.1\sqrt{n\gamma}\right)_{+}^2\right] \right]\ .
 \end{equation} 
 In comparison to the second moment bound for the original distribution $\P(X)$, the bound for $\P'(X)$ involves a correction factor $\left(\lambda\sqrt{z^2+t^2}/2\gamma \sqrt{n} - 2.1\sqrt{n\gamma}\right)_{+},$ 
 which is most effective when $z$ is large. If instead $z \leq n/(2\eta^2)$, we make use of the bound~\eqref{eq:upper_small_z} proved for second moment bound:
 \begin{equation}
 	\Exp_t \!\left[  \e^{ \frac{ \eta^2 t^2}{2n} }|z\right]\leq O \left( 1+ \frac{\eta^2}{n/z - \eta^2} \right)  = O\left(1\right)\ .\label{eq:upper_2nd_small}
 \end{equation}

 For larger values of $z$, we rely on the fact that the hypergeometric distribution with parameters $(n,\gamma n,\gamma n)$ is stochastically dominated by the Binomial distribution with parameters $(\gamma n , \gamma/(1-\gamma))$.  
 It follows  from Chernoff's bound  that 
 \begin{eqnarray} \nonumber
 \Exp_{z,t}\left[e^{\frac{\eta t^2}{2 n}}\indicator{z\in (\frac{n }{2\eta^2}, \gamma n/3) }\right] &\stackrel{(a)}{\leq} 
  &\sum_{z= n/(2\eta^2)}^{\gamma n/3}\exp\!\left(\frac{\eta^2 z^2 }{2 n} - \gamma n h_{\gamma/(1-\gamma)}\left(\frac{z}{\gamma n}\right) \right) \\ \nonumber
  &\stackrel{(b)}{\leq} & \sum_{z= n/(2\eta^2)}^{\gamma n/3}\exp\!\left(z\left( \frac{\eta^2 z}{2 n } -  \log\left(\frac{z(1-\gamma)}{e\gamma^2 n}\right) \right)\right)\\ \nonumber  
  &\stackrel{(c)}{\leq} & \sum_{z= n/(2\eta^2)}^{\gamma n/3} \exp\!\left(z\left( \left(\frac{\lambda^2}{6\gamma}   - \log\frac{1-\gamma}{3e\gamma}
  \right)\vee \left(\frac{1}{4}+ \log\frac{2\e \lambda^2}{1-\gamma}\right)\right)\right)\\ \nonumber
  &\stackrel{(d)}{\leq} & \sum_{z= n/(2\eta^2)}^{\gamma n/3} \exp\!\left(z\left( \left(\frac{\lambda^2}{6\gamma}   - \log\frac{1-\gamma}{3e\gamma}
  \right)\vee \left( \log\frac{8\e^{5/4} \gamma \log(1/\gamma)}{1-\gamma}\right)\right)\right)\\
  & \overset{}{\leq} & \sum_{z= n/(2\eta^2)}^{\gamma n/3}\exp\!\left(- z  \epsilon \right) =o(1) \ , \label{eq:upper_2nd_medium}
 \end{eqnarray}
where the last line holds for some constant $\epsilon>0$ in view of \prettyref{eq:cond_lambda_small_gamma} and of the condition $\gamma<1/(81e^{41})$.
We used in $(a)$ that $t\leq z$, in $(b)$ that $h_{p_0}(p_1)\geq p_1\log(p_1/(\e p_0))$ in $(c)$ that the function $x\mapsto \eta^2 x /(2n)-\log(x)$ is decreasing on $(0,2n /\eta^2]$ and increasing on $[2n/\eta^2,\infty)$. and in $(d)$ that $\lambda^2 \le -4 \gamma \log \gamma$.

Finally, we turn to the case $z\in [\gamma n/3,\gamma n]$, for which we will rely on the correcting term in \eqref{eq:upper_second_moment}. 
 Assume that $\lambda\geq 9\sqrt{\gamma}$ without loss of generality, because otherwise by \prettyref{eq:gamma_regime2}, the second moment is bounded without any conditioning.
For $z$ larger than $\gamma n/3$, 
$$
	\frac{\lambda z }{\sqrt{2}\gamma \sqrt{n}} \ge 2.1\sqrt{n\gamma},
$$ 
and we can once again apply Chernoff's bound to obtain
\begin{eqnarray*} 
	\lefteqn{\Exp_{z,t}\left[\exp\!\left(\frac{\eta^2 t^2}{2 n} -\frac{1}{2}\left(\frac{\eta\sqrt{z^2+t^2} }{2 \sqrt{n}}- 2.1\sqrt{n\gamma}\right)_{+}^2\right)\indicator{z\in (\gamma n/3, \gamma n]} \right] }&&\\
	&\leq& \Exp_{z,t}\left[\exp\!\left(\frac{\eta^2 z^2}{2 n} -\frac{1}{2}\left(\frac{\eta z }{ \sqrt{2n}}- 2.1\sqrt{n\gamma}\right)_{+}^2\right)\indicator{z\in (\gamma n/3, \gamma n]} \right] \\
	&\leq & \sum_{z=\gamma n/3}^{\gamma n}  \exp\!\left( \frac{\eta^2 z^2}{4n} + \frac{2.1}{\sqrt{2}}\eta z\sqrt{\gamma }-  \gamma n h_{\gamma/(1-\gamma)} \left(\frac{z}{\gamma n} \right)\right) \\
	&\stackrel{(a)}{\leq }& \sum_{z=\gamma n/3}^{\gamma n} \exp\!\left( \frac{\eta^2z^2 }{4n} + \frac{2.1}{\sqrt{2}}\eta z\sqrt{\gamma }-  z\log\left(\frac{z(1-\gamma)}{e\gamma^2 n }\right)\right) \\
	& \leq & \sum_{z=\gamma n/3}^{\gamma n} \exp \left(z\left( \frac{\eta^2 \gamma}{4}    - \log\left(\frac{1}{\gamma} \right) + \frac{ 2.1 \eta\sqrt{\gamma}  }{\sqrt{2}}     + \log\left(\frac{3\e}{1-\gamma}\right ) \right)\right)\\
	&\leq & \sum_{z=\gamma n/3}^{\gamma n} \e^{-z \epsilon}= o(1) \ ,
 \end{eqnarray*} 
where we used $t \le z$ and the fact that the whole exponent is  monotone increasing in $t^2$ in the first line. The last inequality holds for some constant $\epsilon>0$, because of condition \prettyref{eq:cond_lambda_small_gamma} on $\lambda$ together with the inequality $\lambda \le 2 \sqrt{\gamma \log (1/\gamma)}$. In $(a)$, we applied   $h_{p_0}(p_1)\geq p_1\log(p_1/(\e p_0))$. Together with \eqref{eq:upper_second_moment}, \eqref{eq:upper_2nd_small}, and \eqref{eq:upper_2nd_medium}, we conclude that $\Exp_{X\sim \mathbb{Q}}\left[\frac{\mathbb{P}'^2(X)}{\mathbb{Q}^2(X)}\right] =O(1)$.

\paragraph{Conditional second moment bound for reconstruction}
In view of \prettyref{eq:secondmoment_KL} and a bounded conditional second moment, we 
have $D_{\mathrm{KL}}(\P' \|Q) = O(1)$. Also, we have $\|M\|_\ast=\|M\|_2=\frac{\lambda}{\sqrt{n}} \|v\|_2^2 = \lambda \sqrt{n}$.
Therefore, the conclusion for reconstruction follows by applying~\prettyref{thm:MMSE}
together with \prettyref{thm:conditional_KL}.

\subsection{Submatrix Localization}

In this section we prove~\prettyref{thm:submatrix}. 
Recall that $X = (\mu/\sqrt n)\left(Y - \allones/k \right) + W$, where $\planted : [n] \to [k]$ is a balanced partition, 
$Y_{i,j} = \indicator{\planted(i) = \planted(j)}$, and $W$ is Wigner.

\subsubsection{First moment upper bound for submatrix localization} 
 It follows from~\prettyref{eq:conditional-ratio} that the conditional log likelihood ratio reads
\begin{align*}
	\log \frac{\P(X | \sigma)}{\Q(X)} 
	= \frac{\mu}{2\sqrt n}\langle Y - \allones /k, X \rangle - \frac{\mu^2}{4n} \|Y - \allones /k \|^2_F.
	%
	%
\end{align*}
Therefore, maximizing $\log \P(X | \sigma)/\Q(X) $ over $\sigma$ is equivalent to computing
$$
	\max_{\sigma} \score(\sigma) := \max_{\sigma} \sum_{i < j} X_{i,j}\left( Y_{i,j} - 1/k  \right). 
$$

In both the planted model and the null model, $\score(\sigma)$ is the sum of independent, Gaussian random variables and is therefore itself Gaussian. Under $\Q$, 
%
\begin{align*}
	\score(\sigma) \sim \calN\left(0, \frac{n^2(k-1) }{2k^2} + O(n) \right).
\end{align*}
Under $\P$, denoting by $\sigma_0$ the planted partition, 
\begin{align*}
	\score(\sigma) & = \frac{\mu}{\sqrt{n}} \sum_{i<j} \left( \indicator{\planted(i)  = \planted(j)} - 1/k \right)  \left(  \indicator{\sigma(i)=\sigma(j)} - 1/k \right) +  \sum_{i < j} W_{i,j}\left(  \indicator{\sigma(i)=\sigma(j)} - 1/ k\right) \\
	&=  \frac{\mu}{\sqrt{n}} \sum_{i<j}  \indicator{\planted(i) = \planted(j),\, \sigma(i)=\sigma(j)} +  \sum_{i < j} W_{i,j}\left(  \indicator{\sigma(i)=\sigma(j)} - 1/k\right) -\frac{n^2}{2k^2} + O(n). 
\end{align*}
Hence the distribution of $\score(\sigma)$ depends on the overlap matrix $\omega$ between $\sigma$ and the planted partition $\sigma_0$. 
%
For a given $\omega$,  among all $\binom{n}{2}$ pairs $i,j$ with $i<j$, 
there are $ \norm{\omega}^2_F n^2/2k^2 + O(n)$ paris such that $\planted(i) = \planted(j)$ and $\sigma(i)=\sigma(j)$.
Therefore, 
\begin{align*}
	\score(\sigma) \sim \calN\left(\frac{\mu n^2 ( \norm{\omega}^2_F - 1) }{2k^2\sqrt n} +O(n) , \frac{n^2(k-1)}{2k^2} +O(n) \right).
\end{align*}
%


To prove that detection is possible, notice that in the planted model, $\max_\sigma \score(\sigma) \ge \score(\sigma_0)$.
Setting $\omega = \identity$, Gaussian tail bounds tell us that
\begin{align*}
	\P\left[\score(\sigma_0) > \frac{n^2\mu (k-1)}{2k^2\sqrt n} - O(n \sqrt{\log n} )\right] \le n^{-\Omega(1)}.
\end{align*}
In the null model,  taking the union bound over the $n^k$ ways to choose $\sigma$, we can bound the probability that \emph{any} partition is as good, according to $\score$, as the planted one, by
\begin{align*}
	\Q\left[\max_{\sigma}\score(\sigma) > \frac{n^2\mu (k-1)}{2k^2\sqrt n} - O( n \sqrt{\log n})\right] \le \exp \left( n\left(\log k - \frac{\mu^2(k-1)}{4k^2} + O(\sqrt{\log n/n}  ) \right) \right) .
\end{align*}
Thus the probability of this event is $\e^{-\Omega(n)}$ whenever
\[
	\mu^2 > \frac{4k^2\log k}{k-1},
\]
meaning that above this threshold we can distinguish the null and planted models with generalized likelihood testing.

To prove that reconstruction is possible, we compute in the planted model the probability that $\score(\sigma) > \score(\planted)$ given that $\sigma$ has $L_2$ overlap $z$ with the planted partition, and argue that this probability tends to zero whenever the overlap is small enough. Taking the union bound over every $\sigma$ with $L_2$ overlap at most $z$ gives
%
\begin{align*}
	\P\left[\max_{\norm{\omega(\sigma,\planted)}^2_F \le z}\score(\sigma) \ge \frac{n^2\mu (k-1)}{2k^2\sqrt n} - O(n \sqrt{\log n} )\right]
	%
	%
	&\le \exp \left(  n \left(\log k - \frac{\mu^2(k - z)^2}{k^2(k-1)} + O(\sqrt{\log n/n} ) \right) \right).
\end{align*}
By the assumption that $\mu^2 > 4k^2\log k/(k-1)$, it follows that there exists a fixed constant $\epsilon>0$ such that $\mu^2 (1-\epsilon)^2 >4k^2  \log k/(k-1)$. Hence, setting $z= 1+ (k-1)\epsilon$ in the last displayed equation, 
it yields that with probability at least $1-e^{-\Omega(n)}$, 
$$
\max_{\norm{\omega(\sigma,\planted)}^2_F \le z}\score(\sigma)  <  \frac{n^2\mu (k-1)}{2k^2\sqrt n} - O(n \sqrt{\log n} ), 
$$
and consequently $ \| w( \hat{\sigma}_{\rm ML}, \sigma_0 ) \|_F^2  \ge 1+ (k-1)\epsilon$. By \prettyref{lmm:overlap}, this further implies that the trace overlap satisfies $ T\left(\hat\sigma_{\rm ML},\sigma_0\right) \ge 1+ (k-1)\epsilon $. Moreover, construct an $n\times n$ matrix $\hat{M}$ so that $\hat{M}_{ij}= (\mu/\sqrt n) \indicator{\hat{\sigma}_{\rm ML} (i) =   \hat{\sigma}_{\rm ML} (j) }  $. Then $\| \hat{M} \|_F^2=O(n)$, and  
$$
\langle \hat{M}, M \rangle =  \frac{n \mu^2}{k^2} \left( \| w( \hat{\sigma}_{\rm ML}, \sigma_0 ) \|_F^2 -1 \right) + O(1),
$$
which is at least $ n \mu^2 (k-1) \epsilon /k^2 + O(1)$ with high probability. 
Thus, we can reconstruct the signal matrix $M$ and the planted partition $\planted$ better than chance.

\subsubsection{Second moment lower bound for submatrix localization} 
We first prove that when 
\begin{equation}
	\mu^2 < 
	\begin{cases}
	k^2 & k=2 \\
	\frac{2k^2\log(k-1)}{k-1} &  k \ge 3 \, ,
	\end{cases} 
	\label{eq: submatrix second moment}
\end{equation}
then the second moment is bounded. Applying~\prettyref{lmm:secondmomentGaussian},
$$
	\Exp_{X \sim \Q} \left( \frac{\P(X)} {\Q(X)} \right)^2  
	= \Exp_{\sigma, \tau} \exp\!\left( \frac{\mu^2}{2 n} \inner{Y - \allones/k}{ Y' - \allones/k } \right),
$$
where $Y_{ij} = \indicator{\sigma(i)=\sigma(j)}$ and $Y'_{ij} = \indicator{\tau(i)=\tau(j)}$. Recall that $\omega$ denotes the overlap matrix between partitions $\sigma$ and $\tau$. Then $\inner{Y}{Y'} = n^2 \| \omega\|_F^2/k^2$ and $\inner{Y}{\allones}=\inner{Y'}{\allones} = n^2/k$. It follows from the last displayed equation that 
\begin{equation}\label{eq:second_moment_submatrix}
	\Exp_{X \sim \Q} \left( \frac{\P(X)} {\Q(X)} \right)^2  
	= \Exp_{\sigma, \tau} \, \exp\!\left( \frac{\mu^2 n}{2 k^2 } ( \| \omega \|_F^2 -1 )\right) ,
\end{equation}
Lemmata \ref{lmm:laplacemethod}, \ref{lmm:A-N} and \ref{lmm:A-N_k_two} assure us that this expression is bounded by a constant so long as \prettyref{eq: submatrix second moment} holds. \\

For reconstruction, in view of \prettyref{eq:secondmoment_KL} and a bounded second moment,
$D_{\mathrm{KL}}(\P \|Q) = O(1)$. 
Apply \prettyref{thm:MMSE} so that for any estimator $\hat{M}$ with $\|\hat{M}\|_F^2 = O(n)$, we have  $\Exp[\inner{M}{\hat{M}}] = o(n)$. Recall that 
$\sigma_0$ is the planted partition and $ M = (\mu/\sqrt n)(Y - \allones/k) = (\mu/\sqrt n)UU^\dagger, $ where $U\in\R^{n\times k}$ with $U_{i s} = \indicator{\sigma_0(i) =s } - 1/k$. 
Then for any estimator $\hat{\sigma}$ of $\sigma_0$,  
by defining $\hat{U}$ such that $\hat{U}_{is} = \indicator{\hat{\sigma}(i)=s} -1/k$ and letting $\hat{M} = (\mu/\sqrt n)\hat{U}\hat{U}^\dagger,$ it follows that 
\begin{align*}
	\frac{1}{n^2}  \Exp \| U^\dagger \hat{U} \|_F^2  =  \frac{1}{\mu^2 n } \Exp \inner{M}{\hat{M}}  = o(1). 
\end{align*}
Finally, notice that by letting $\omega= \omega( \hat\sigma ,\sigma_0) $, 
$$
\| U^\dagger \hat{U} \|_F^2  = \frac{n^2}{k^2} \sum_{\ell, s} (\omega_{\ell, s} - 1/k)^2 = \frac{n^2}{k^2} \left(  \| \omega\|_F^2 -1 \right);
$$
thus $\Exp  \| \omega\|_F^2=1+o(1)$. 
Hence, we can neither reconstruct the signal matrix $M$ nor the planted partition $\sigma_0$.

\subsubsection{Conditional second moment lower bound for submatrix localization} \label{sec:conditional_submarix}

Notice that in the large $k$ asymptotic, the right hand side of \prettyref{eq: submatrix second moment} converges
to $2 k \log k$, while the first moment upper bound gives $4k \log k$. To match the first moment upper bound in the larger $k$ asymptotic, we apply a conditional second moment method. In the sequel we assume that
\begin{equation}
\label{eq:condition_mu_large_k}
 \mu^2 \leq 4k\log k-44 k\log^{3/4}(k) .
\end{equation}
 If $\log k <11^4$, then  $4k\log k - 44 k \log^{3/4} k$ is negative. Hence, we focus 
 on the setting $\log k \ge 11^4$.

For any partition $\sigma: [n] \to k$,  let $M=\frac{\mu}{\sqrt{n}}(Y_{\sigma}- \mathbb{J}/k) $
and
define the event $\Gamma_{\sigma}$  by 
\[
\Gamma_{\sigma}:= \left \{\sup_{1 \le \ell \le k } \left \| \left [X - M\right]_{\sigma^{-1}(\ell)} \right \|_2\leq 3 \sqrt{n/k} \right\},
\]
where $ \left [X -M \right]_{\sigma^{-1}(\ell )}$ denotes the $\sigma^{-1}( \ell )\times \sigma^{-1}( \ell )$ submatrix of $X-M$.
In other words, on $\Gamma_\sigma$, the spectral noms of submatrices $ \left [X -M \right]_{\sigma^{-1}(\ell)}$
for every $1 \le \ell \le k$ are all upper bounded by $3 \sqrt{n/k}$.
Since $ \left [X -M \right]_{\sigma^{-1}(\ell)}$ is distributed as a Wigner Matrix under $\mathbb{P}(.|\sigma)$,
by Gaussian concentration theorem and Slepian's inequality we have
 \[
  \mathbb{P} \left[ \left \| \left[X- M\right]_{\sigma^{-1}(\ell)} \right \|_2\geq 2\sqrt{n/k}+ t  \mid \sigma \right]\leq \e^{-t^2/4}\ , \forall 1 \le \ell \le k.
 \]
Since $k=o(n/\log(n))$, we have $\mathbb{P}(\Gamma_\sigma|\sigma)=1+o(1)$ uniformly with respect to  $\sigma$.\\

Next, we condition on the high probability events $\Gamma_\sigma$ by defining
 \[
  \mathbb{P}'(X)= \mathbb{E}_{\sigma}\left[\mathbb{P}(X|\sigma)\frac{\indicator{\Gamma_\sigma}}{\mathbb{P}(\Gamma_\sigma|\sigma)} \right],
 \]
It follows that the conditional second moment  satisfies 
 \begin{eqnarray*} 
  \Exp_{X\sim \mathbb{Q}}
  \left[ \left( \frac{\mathbb{P}'(X)}{\mathbb{Q}(X)} \right)^2\right]&=& \Exp_{\sigma,\tau}
  \left[\Exp_{X\sim \mathbb{Q} }  
  \left[ \frac{\mathbb{P}(X|\sigma)\mathbb{P}(X|  \tau)}{\mathbb{Q}^2(X)} \frac{ \indicator{\Gamma_{\sigma} }\indicator{\Gamma_{\tau}}} {\mathbb{P}( \Gamma_{\sigma} | \sigma)\mathbb{P}(\Gamma_{\tau} | \tau)}\right]\right] \\
   &= &(1+o(1))\mathbb{E}_{\sigma,\tau}\left[\underbrace{\Exp_{X\sim \mathbb{Q}}\left[e^{\frac{\mu}{2\sqrt{n}}\langle X , Y_{\sigma}+ Y_{\tau}-2 \mathbb{J}/k \rangle - \frac{\mu^2}{2n}\|Y_{\sigma}-\mathbb{J}/k\|_F^2}\indicator{\Gamma_{\sigma}}\right] } \right]\ ,
   \\ && \hspace{5cm}=(I)
 \end{eqnarray*}   
 where we used that $\mathbb{P}(\Gamma_v|v)=1+o(1)$.  
 Given $\sigma$ and $\tau$, define $\mathcal{A}_{\ell \ell'}= \sigma^{-1}(\ell )\cap\tau^{-1}(\ell')$ and the event 
\[
\overline{\Gamma}_{\ell \ell'}= 
\left\{
\sum_{i, j \in \mathcal{A}_{\ell\ell'}} X_{ij} \leq \frac{\mu}{\sqrt{n}} \left(1-\frac{1}{k} \right) |\mathcal{A}_{\ell\ell'}|^2+ 3|\mathcal{A}_{\ell \ell'}| \sqrt{n/k}
\right\}
\]
Obviously,  $\overline{\Gamma}_{\ell \ell'} \subset\Gamma_{\sigma}$ for all $\ell, \ell'\in [k]$. Equipped with this notation, we bound $(I)$ by the exponential moment of $k^2+1$  thresholded normal random variables:
\begin{eqnarray*} 
(I)
&\leq  & \mathbb{E}_{X\sim \mathbb{Q}}\left[ \exp \left( \frac{\mu}{2\sqrt{n}}\langle X , Y_{\sigma}+ Y_{\tau}-2 \mathbb{J}/k \rangle - \frac{\mu^2}{2n}\|Y_{\sigma}-\mathbb{J}/k\|_F^2 \right) \prod_{\ell,\ell'=1}^k\indicator{\overline{\Gamma}_{\ell \ell'}}\right]\\
&\leq & \exp\!\left( \frac{\mu^2}{2n} \langle Y_{\sigma}-\mathbb{J}/k, Y_{\tau}-\mathbb{J}/k\rangle \right) 
\prod_{\ell,\ell'=1}^k \exp\!\left( - \frac{1}{2} \left[
\mu\left(1-\frac{1}{k} \right)\frac{|\mathcal{A}_{\ell \ell'}|}{\sqrt{2n}}- 3\sqrt{\frac{n}{2k}} \right]_{+}^2 \right)
\end{eqnarray*} 
where we used  
$$
\mathbb{E}_{Z\sim \mathcal{N}(0,1)}[e^{aZ}\indicator{Z\leq b}] \leq \e^{a^2/2- (a-b)^2_+/2}.
$$ 
Recall that $\omega$ denotes the overlap matrix between $\sigma$ and $\tau$. 
Then $\langle Y_{\sigma},Y_{\tau}\rangle=n\|\omega\|_F^2/k^2$, $\langle Y_{\sigma},\mathbb{J}\rangle = \langle Y_{\sigma},\mathbb{J}\rangle =n^2/k$ and $|\mathcal{A}_{\ell \ell'}|= n\omega_{\ell\ell' }/k$. 
Define 
\begin{align*}
a_k:= & 3\log^{1/4}(k)  \\
b_k:= & 2/k + 6/a_k.
\end{align*}
It follows from the last displayed equation that 
\begin{eqnarray}\nonumber
 \frac{\log(I)}{n} &\leq &\frac{\mu^2}{2k^2}(\|\omega\|_F^2 -1)- \frac{\mu^2}{4k^2}\sum_{\ell,\ell'}\omega_{\ell\ell'}^2\left(1-\frac{1}{k} - 3\frac{\sqrt{k}}{\mu \omega_{\ell\ell'}} \right)_+^2  \label{eq:upper_logII} \\ \label{eq:upper_logI}
&\leq  & \frac{\mu^2}{4k^2}(1+ b_k)(\|\omega\|_F^2 -1) +  \frac{\mu^2}{4k^2}\sum_{\ell,\ell'} \left(\omega_{\ell\ell'}^2-\frac{1}{k^2} \right) \indicator{\omega_{\ell\ell'}\leq \sqrt{k} a_k/\mu}\ ,
\end{eqnarray}
where we used $(1-x)^2_+\geq \max((1-2x),0)$  and $b_k \le 1$ since  $\log k\geq 11^4$ in the second line. Unfortunately, the expression 
$$
\exp \left(\sum_{\ell,\ell'} \left(\omega_{\ell\ell'}^2-\frac{1}{k^2} \right) \indicator{\omega_{\ell\ell'}\leq \sqrt{k} a_k/\mu} \right)
$$ 
is not easy to integrate over $\omega$. This is why we shall bound it by a combination of the entropy of $\omega$ and a second degree polynomial.

\begin{lemma}
Let $t\in (0,1)$ be such that $t\leq 0.4$ and $t\geq \tfrac{3}{k}\vee (\tfrac{\e}{k})^{1/5}\vee \frac{1-1/k}{2[\log k-1]}\vee  \frac{2 \log k}{5 (k- 2 \log k +1 ) }$. Upon defining 
$c_t:= \frac{5t}{\log k-1} - \frac{2+ 5t }{k[\log k-1]}>0$, we have for any $x\in [0,1]$, 
\begin{equation}\label{eq:ineq:fun}
(x^2-\frac{1}{k^2})\indicator{x<t}\leq 5t(x-\frac{1}{k})(1-x) +c_t \left[x\log(x) + \frac{\log k}{k}\right]\ .
\end{equation} 
\end{lemma}
\begin{proof}
 Indeed, let $g$ be defined  by $g(x)= 5t(x-\frac{1}{k})(1-x) +c_t \left[x\log(x) + \frac{\log k}{k}\right]- (x^2-\frac{1}{k^2})\indicator{x<t}$. For any $x\neq t$, 
$g'(x)= 5t(-2x +1+1/k) +c_t\log(xe) -2x\indicator{x<t}$ and $g''(x)=  \frac{c_t}{x} -10t-2\cdot\indicator{x<t}$. Since $c_t\leq 10t^2$,  $g'$ is increasing on $(0,c_t/[2+10t])$ and decreasing on $(c_t/[2+10t],t)$ and $(t,1]$. Besides, $c_t/[2+10t] \ge 1/k$ as this inequality reduces $t \ge \frac{2 \log k}{5 (k- 2 \log k +1 ) }.$
Since $c_t$ has been chosen in such a way that $g'(1/k)=0$, the minimum of $g$ is either achieved at $1/k$, $t_-$, or $1$.  We have $g(1/k)=0$ and $g(1)=c_t\log k/k>0$, and 
\[g(t_-)\geq 5t(t-1/k)(1-t)-c_t t\log(1/t)-t^2\geq t^2 - \frac{5t^2\log(1/t)}{\log k-1}\geq 0\ ,\]
where we used $(t-1/k)(1-t)\geq 2t/5$ and $t\geq (e/k)^{1/5}$.
We have proved \eqref{eq:ineq:fun}. 
\end{proof}

Since $\mu\leq 2\sqrt{k\log k}$, we have $\sqrt{k}a_k/\mu \ge \frac{3}{2} \log^{-1/4}(k)$. Assuming that $\mu \geq 2.5\sqrt{k}a_k$ and since $\log k\geq 11^4$, the assumptions of \eqref{eq:ineq:fun}
are satisfied  for $t=\sqrt{k}a_k/\mu$. Writing $c_k$ for $c_{\sqrt{k}a_k/\mu}$, it follows from \eqref{eq:upper_logI} and \eqref{eq:ineq:fun} that 
\begin{eqnarray}\nonumber
 \frac{\log(I)}{n} &\leq &
 \frac{\mu^2}{4k^2}(1+ b_k)(\|\omega\|_F^2 -1) +  \frac{5\mu a_k}{4k^{3/2}}\sum_{\ell,\ell'} \left(\omega_{\ell\ell'}-\frac{1}{k} \right)(1-\omega_{\ell\ell'}) +\frac{\mu^2c_k}{4k^2}\sum_{\ell,\ell'} \left[\omega_{\ell\ell'}\log(\omega_{\ell\ell'})+\frac{\log k}{k} \right]\\
 &\stackrel{(a)}{\leq} & \frac{\mu^2}{4k^2}(1+ b_k)(\|\omega\|_F^2 -1)-\frac{\mu^2c_k}{4k} \left[H(\omega) -\log k\right]\nonumber \\
 & \leq  & \frac{\mu^2}{4k^2}(1+ b_k)(\|\omega\|_F^2 -1)-\frac{5 \mu a_k}{4\sqrt{k}(\log k-1)} \left[H(\omega) -\log k\right] \nonumber\\
 &\leq & \frac{\mu^2}{4k^2}(1+ b_k)(\|\omega\|_F^2 -1)- \tfrac{15}{2}\log^{-1/4}(k) \left[H(\omega) -\log k\right]\ ,\nonumber
\end{eqnarray}
where $(a)$ follows because  $\omega$ is doubly stochastic so that
$\sum_{\ell,\ell'}(\omega_{\ell\ell'}-\frac{1}{k})(1-\omega_{\ell\ell'})=1-\|\omega\|_F^2 \le 0$;
the last inequality holds due to $\mu \leq 2\sqrt{k\log k}$. 
If  $\mu \leq 2.5\sqrt{k}a_k$,
we simply come back to  \eqref{eq:upper_logII} to ensure that  $\log(I) \leq n \frac{\mu^2}{2k^2}(\|\omega\|_F^2 -1)$. We arrive at
\begin{align}
\Exp_{X\sim \mathbb{Q}}
  \left[ \left( \frac{\mathbb{P}'(X)}{\mathbb{Q}(X)} \right)^2\right] \leq (1+o(1))\Exp_{\sigma,\tau}\exp \left[n \left( t_1(\|\omega\|_F^2 -1) -  t_2 \left( H(\omega) -\log k\right) \right)\right] \label{eq:submatrix_conditional_second_moment}
\end{align}
with $t_1=\frac{\mu^2}{2k^2}$ and $t_2=0$ when $\mu < 2.5\sqrt{k}a_k$ and $t_1= \frac{\mu^2}{4k^2}(1+ b_k)$  and  $t_2= 8\log^{-1/4}(k)$ when $\mu\geq 2.5\sqrt{k}a_k$. 

Hence, to prove the conditional second moment is bounded, it reduces to verifying the
right hand side of  \prettyref{eq:submatrix_conditional_second_moment} is bounded. 
 By assumption \eqref{eq:condition_mu_large_k} and $ \log k \ge 11^4$, it holds that 
$$
\mu^2 \le \frac{4k^2}{(k-1)(1+b_k)} \left( \log (k-1) - 8 \log (k-1) \log^{-1/4} (k) -1 \right)\ .
$$
This further implies that 
\[
 \frac{t_1}{1-t_2}\leq \frac{\log(k-1) -1/(1-t_2)}{k-1}\ .
\]
Therefore, by \prettyref{lmm:A-N}, we have that for all $\omega$,
$$
\left( H(\omega)-\log k\right)(1-t_2)   + t_1 \left( \|\omega\|_F^2 -1\right)  \  \le -\frac{1}{k-1}(\|\omega\|_F^2 -1).
$$
Hence, assumptions in \prettyref{lmm:laplacemethod} are satisfied with $\varphi(\omega)= t_1(\|\omega\|_F^2 -1) -  t_2 [ H(\omega) -\log k ]$
and $\delta= 1/(k-1)$, 
and it follows from  \prettyref{lmm:laplacemethod} that the right hand side of \prettyref{eq:submatrix_conditional_second_moment} is bounded.

\paragraph{Conditional second moment bound for reconstruction}
 In view of \prettyref{eq:secondmoment_KL} and a bounded conditional second moment, we 
have $D_{\mathrm{KL}}(\P' \|Q) = O(1)$. 
Also, we have that 
$$
\|M\|_2= \frac{\mu}{\sqrt{n}} \| Y - \allones /k \|_2 \le \frac{\mu}{\sqrt{n}} \left( \|Y\|_2 + \| \allones \|_2 /k \right)=2 \mu \sqrt{n}/k
$$
and since $M$ is of rank $k$, $\|M\|_\ast \le k \|M\|_2 \le 2 \mu \sqrt{n}$. 
Therefore, the conclusion for reconstruction follows by applying~\prettyref{thm:MMSE}
together with \prettyref{thm:conditional_KL}.


\subsection{Gaussian Mixture Clustering}
Finally, we turn to the Gaussian Mixture Clustering problem, where $\sigma : [n] \to [k]$ is a balanced partition chosen uniformly at random, 
$v_1,\ldots, v_k \iiddistr \calN(0, k/(k-1)\identity_{n,n})$, and we observe the matrix
\begin{align}
	X = \sqrt{\frac{\rho}{n}}\left(S - \frac{1}{k}\allones_{m,k}\right)V^\dagger + W, \label{eq:Gaussianmixtuer_model}
\end{align}
where $W$ has independent standard normal entries, $S$ is an $ m \times k$ matrix with $S_{i,t} = \indicator{\sigma(i) = t}$, and $V = [v_1, \dots ,v_k]$.

\subsubsection{First moment upper bound for Gaussian mixture clustering}

In this section, we derive an upper bound on the detection and reconstruction threshold via the first moment method. The testing procedure is again based on the generalized likelihood ratio $\sup_{v, \sigma} \log  \P(X | v, \sigma )/\Q(X)$, which from~\prettyref{eq:conditional-ratio} we write as
\begin{align*}
	\log \frac{\P(X \mid v,\sigma)}{\Q(X)} 
	&=\sqrt{ \frac{\rho}{ n} } \langle X, (S - \allones_{m,k}/k)V^\dagger \rangle - \frac{\rho}{2n}\norm{(S - \allones_{m,k}/k)V^\dagger}_F^2 \\ 
	& = \sqrt{ \frac{\rho}{ n} }\sum_{s=1}^k \inner{\sum_{i: \sigma(i)=s } x_i }{ v_s - \overline{v} } - \frac{\rho m}{2nk} \sum_{s=1}^k \|v_\ell -\overline{v} \|_2^2.
\end{align*}
%
When $\sigma$ is fixed, we can optimize in $v$ by setting the $s$th cluster center to
$$
	v_s -\overline{v} = \frac{ k \sqrt{n} }{m \sqrt{\rho} } \sum_{i: \sigma(i)=s } x_i , 
$$ 
the rescaled center of the data points $x_i$ which have been assigned to cluster $s$ according to $\sigma$. Up to multiplicative constants, then, the generalized likelihood ratio test is equivalent to 
the test based on the statistic
\begin{align} \label{eq:clustering_test_statistic}
	\max_\sigma \score(\sigma) := \max_\sigma \frac{k}{m} \sum_s \norm{\sum_{\sigma(i) = s} x_i}^2.
\end{align}
In the null model, for any fixed $\sigma$,
\begin{align*}
	\score(\sigma) \sim \chi^2_{nk},
\end{align*}
where $\chi^2_{nk}$ is the central chi-squared distribution with $nk$ degrees of freedom.

In the planted model, let $v^0$ and $\sigma_0$ denote the planted vectors and partition respectively, and for an arbitrary partition $\sigma$, let $\omega$ once again be the overlap matrix between $\sigma$ and $\sigma_0$. 
%
For any $1 \le \ell \le k$,
\begin{align*}
	\sum_{\sigma(i) = s} x_i \sim \calN\left(\frac{m\sqrt{\rho}}{k\sqrt n}\sum_t \omega_{s,t} (v^0_t - \overline v^0), \frac{m}{k}\identity\right).
\end{align*}
For $y\sim\calN(\mu,\identity_d)$, let $\chi^2_d(\| \mu\|^2)$ denote the distribution of $\|y\|^2$, which is known as non-central chi-square distribution with $d$ degrees of freedom and non-centrality $\|\mu\|^2$.  
In this notation, $\score(\sigma)$ is distributed in $\P$ as a non-central chi-squared  random variable with $nk$ degrees of freedom and noncentrality
\begin{align}
	\frac{\alpha \rho}{k}\sum_{\ell=1}^k  \norm{\sum_s \omega_{\ell,s} (v^0_s - \overline v^0)}^2 
	&= \frac{\alpha\rho}{k}\sum_{\ell=1}^k  \sum_{s,t}\omega_{\ell,s}\omega_{\ell,t}\langle v^0_s - \overline v^0, v^0_t - \overline v^0 \rangle  \nonumber \\
	&= \frac{n \alpha \rho}{k-1}\left(\norm{\omega}^2_F - 1\right) + O(\sqrt{n\log n}). \label{eq:signal_gaussian}
\end{align}
To obtain the last line, note that the $v^0_\ell$ are Gaussian with zero mean and variance $k/(k-1)$ in each coordinate, making $\langle v^0_s - \overline v^0, v^0_t - \overline v^0 \rangle = - n/(k-1) + O(\sqrt{n\log n})$ for $s \neq t$ and $\|v^0_s - \overline v^0\|^2 =  n + O(\sqrt{n\log n})$ with high probability. 

We will need the following tail bounds for non-central chi-squared distributions~\cite[Lemma 8.1]{birge2001}: for $t>0$, 
\begin{align}
	\Pr\left[\chi^2_d\left(\norm{\mu}^2\right) < d + \norm{\mu}^2 - 2\sqrt{\left(d + 2\norm{\mu}^2\right)t}\right] < \e^{-t} \label{eq:chi_lower} \\
	\Pr\left[\chi^2_d\left(\norm{\mu}^2\right) > d + \norm{\mu}^2 + 2\sqrt{\left(d + 2\norm{\mu}^2\right)t} + 2t\right] < \e^{-t} \label{eq:chi_upper}
\end{align}
Notice that we can obtain central chi-square tail bounds by setting $\norm{\mu}^2 = 0$.

To derive the first moment bound for detectability, notice that in the planted model, when $\sigma = \sigma_0$ up to a permutation of cluster indices, $\norm{\omega}^2_F = k$ and
\begin{align*}
	\score(\sigma_0) \sim \chi^2_{nk}\left(\alpha n\rho + O(\sqrt{n\log n})\right).
\end{align*}
Setting $t=\log n$ in~\prettyref{eq:chi_lower}, we know that with high probability $\score(\sigma_0) > nk + n\alpha \rho - O(\sqrt{n\log n})$.  

In the null model, by the union bound and~\prettyref{eq:chi_upper}, letting  $t= (1+\epsilon) m \log k$ for an arbitrarily small constant $\epsilon>0$, 
\begin{align*}
	\Q\left[\max_\sigma \score(\sigma) > nk + 2\sqrt{nkt} + 2t \right] 
	%
	%
	= k^m \Pr\left[\chi^2_{nk} > nk + 2\sqrt{nkt} + 2t\right] = \e^{-\Omega(m) },
\end{align*}
%
%
%
Hence, with high probability, in the null model, 
$$
	\max_\sigma \score(\sigma) \le  nk+ 2\sqrt{(1+\epsilon) mn  k \log k} + 2 (1+\epsilon) m \log k.
$$
By the assumption that $ \rho> \rho^{\up}$, \ie, $ \rho \sqrt{\alpha} > 2 \sqrt{k \log k} + 2 \sqrt{\alpha} \log k$, it follows that for sufficiently large $n$, 
$$
	n \alpha    \rho - O \left(  \sqrt{ n \log n} \right ) \ge 2 (1+\epsilon)    \left( \sqrt{ m n k \log k} + m \log k \right), 
$$
and consequently, with high probability, $ \max_{\sigma} \score(\sigma) > nk + n\alpha \rho - O(\sqrt{n\log n})$ under $\P$, and $\max_{\sigma} \score(\sigma) < nk + n\alpha \rho - O(\sqrt{n\log n})$ under $\Q$. 
%

To show reconstruction is possible above this bound, let $\sigma$ have overlap matrix $\omega$ with $\sigma_0$ and 
set $\|\omega\|^2_F = \theta$. We will show that, when $\theta$ is sufficiently small, there are with high probability no such partition with likelihood as high as the planted one. In the planted model,
\begin{align*}
	\score(\sigma) \sim \chi^2_{nk}\left( n\alpha\rho\frac{\theta - 1}{k-1} + O(\sqrt{n\log n})\right).
\end{align*}
Taking the union bound over all partitions $\sigma$ which have $L_2$ overlap at most $1 + (k-1)\epsilon$ with the planted one---of which there are no more than $k^m$---and invoking~\prettyref{eq:chi_upper} with $t= (1+\epsilon) m \log k$, we know that with high probability, 
\begin{align*}
	\max_{\sigma: \norm{\omega(\sigma,\sigma_0)}^2_F \le   1+(k-1) \epsilon}  \score(\sigma)  &  \le nk +  n\alpha\rho\epsilon+  2 \sqrt{ (1+\epsilon) \left(nk + 2 n\alpha\rho \epsilon \right)m \log k } \\
	&\qquad +  2(1+\epsilon) m \log k + O(\sqrt{n \log n} ),
\end{align*}
By the assumption that  $ \rho \sqrt{\alpha} > 2 \sqrt{k \log k} + 2 \sqrt{\alpha} \log k$, it follows that for sufficiently large $n$, 
$$
	n \alpha    \rho (1-\epsilon) - O \left(  \sqrt{ n \log n} \right ) \ge 2 (1+\epsilon)    \left( \sqrt{ (nk + 2 n \alpha \rho \epsilon ) m  \log k} + m \log k \right), 
$$
 and consequently, with high probability, 
 $$
 	\max_{\sigma: \norm{\omega(\sigma,\sigma_0)}^2_F \le 1+(k-1) \epsilon}  \score(\sigma) < \score(\sigma_0).
 $$ 
 Let $\hat{\sigma}_{\rm ML}=\arg \max \score(\sigma) $ denote the maximum likelihood estimator of $\sigma_0$. Then with high probability, $\norm{\omega(\hat{\sigma}_{\rm ML},\sigma_0)}^2_F \ge 1+ (k-1) \epsilon$, which further implies that the trace overlap satisfies $ T\left(\hat{\sigma}_{\rm ML}, \sigma_0\right)  \ge 1+ (k-1)\epsilon $ in view of \prettyref{lmm:overlap}.

Finally, we argue that above the first moment bound, one can construct an estimator $\hat{M}$ of $M$ such that $\Exp[ \|\hat{M}\|_F^2] = O(n^2)$ and $\Exp[\langle M, \hat{M} \rangle] =\Omega(n)$. Intuitively, if we can estimate the planted partition better than chance, then we should be able to construct an estimator of the signal matrix which out-performs the trivial one. Our proof uses the sample splitting method. Thinking of the rows of $X$ as noisy observations of the cluster centers contained in the rows of $M$, we will project each observation into two orthogonal subspaces, use the first of these projections to build an estimator $\hat \sigma$ of the planted partition $\sigma_0$, and finally combine $\hat \sigma$ with the \emph{second} projection to estimate the cluster centers. This technique gains us a subtle and important independence property: the partition we estimate based on the first projection is independent of the noise in the second.

Let us proceed. For $\delta \in (0,1)$ to be optimized later, let $n_1=(1-\delta) n \in \naturals$ and denote by $X_1 \in \reals^{m \times n_1}$ and $X_2 \in \reals^{m \times (n-n_1)}$ the restrictions of the data matrix $X$ to its first $n_1$ and final $n - n_1$ columns respectively; define $M_1$ and $M_2$ analogously. We first reconstruct $M_1$ from $X_1$ with parameters $\rho' = (\rho/n) n_1 = (1-\delta) \rho$ and $\alpha' = m/n_1= \alpha/(1-\delta)$. Notice that our assumption $ \rho \sqrt{\alpha} > 2 \sqrt{k \log k} + 2 \sqrt{\alpha} \log k$ means that we can choose $\delta$ sufficiently small to ensure $\rho' \sqrt{\alpha'} >  2 \sqrt{k \log k} + 2 \sqrt{\alpha'} \log k$. Let $\hat{\sigma} $ denote the ML estimator of the planted partition based only on the data $X_1$. We have already shown that  $\norm{\omega(\hat{\sigma},\sigma_0)}^2_F \ge 1+ (k-1) \epsilon$ with high probability. 

Now we use $\hat{\sigma}$ to construct an estimator of $M_2$. Specifically, let $\hat S$ be the $m \times k$ indicator matrix for $\hat{\sigma}$, $\hat{S}_{is} = \indicator{\hat{\sigma} (i) =s }$, and define $\hat{M}_2 = (k/m)\hat{S} \hat{S}^\dagger X_2$. Then
$$
	\Exp \inner{M_2}{\hat{M}_2} 
	= \frac{k}{m} \Exp \inner{M_2}{ \hat{S} \hat{S}^\dagger (M_2 + W_2) }
	\overset{(a)}{=} \frac{k}{m} \Exp \inner{M_2}{ \hat{S} \hat{S}^\dagger M_2} 
	=\frac{k}{m} \Exp \| \hat{S}^\dagger M_2\|_F^2,
$$ 
where $(a)$ follows because the noise $W_2$ is independent of $M_2$ and $\hat{S}$ and $\Exp W_2=0.$ Furthermore, 
$$
	\frac{k}{m} \Exp \| \hat{S}^\dagger M_2\|_F^2 = \frac{k}{m} \frac{m^2}{k^2} \frac{\rho}{n} \sum_{\ell=1}^k \left\| \sum_s \omega_{\ell, s} (u_s^0 - \bar{u}^0) \right\|^2,
$$
where $u_s^0$ is the restriction of $v_s^0$ to the last $\delta n$ coordinates, $\bar{u}^0$ is the restriction of $\bar{v}^0$ to the last $\delta n$ coordinates, and $\omega_{\ell, s} = \omega(\hat{\sigma},\sigma_0).$ In view of \prettyref{eq:signal_gaussian}, it follows from the last two displayed equations that $\Exp \inner{M_2}{\hat{M}_2} =\Omega(n)$.
One can additionally verify that that $\Exp \| \hat{M}_2\|_F^2 = O(n^2).$
Finally letting $\hat{M} = [ \mathbf{0}_{m, n_1}, \hat{M}_2]$ be the estimator, \ie, concatenating $\hat M_2$  on the left with $n_1$ zero columns, we have shown that $\Exp \|\hat{M}\|_F^2 = O(n^2)$ and $\Exp \langle M, \hat{M} \rangle =\Omega(n)$. 
 
%
%

\subsubsection{Second moment lower bound for Gaussian mixture clustering}
\label{sec:secondmoment_clustering}

We first show that if  $\rho < \rho^{\low}$, 
\ie, $\alpha\rho^2 < 2(k-1)\log(k-1),$
then the second moment is bounded. 
Recall that $S$ is the $ m \times k$ indicator matrix for a partition $\sigma$ with $S_{i,t} = \indicator{\sigma(i) = t}$, and $V = [v_1, \ldots ,v_k]$. Let $\tau$ be an independent copy of $\sigma$, with indicator matrix $T_{i,t} = \indicator{\tau(i) = t}$ and $U=[u_1, \ldots, u_k]$ be an independent copy of $V$. Applying Lemma 1,
\begin{align*}
	\Exp_{X\sim Q}\left( \frac{\P(X)}{\Q(X)} \right)^2 
	&= \Exp_{\sigma,\tau}\Exp_{v,u} \exp\!\left(\frac{\rho}{n}\left\langle (S - \allones_{m,k}/k)V^{\dagger},(T - \allones_{m,k}/k)U^{\dagger} \right\rangle\right) \\
	&= \Exp_{\sigma,\tau}\Exp_{v,u} \exp\!\left(\frac{\alpha \rho}{k}\langle U,V(\omega - \allones/k) \rangle \right),
\end{align*}
where $\omega$ is the overlap matrix between $\sigma$ and $\tau$. The last equality follows from $S^\dagger T = (m/k)\omega$ and $S^\dagger \allones_{m,k} = (m/k) \allones_{k,k} = \allones_{m,k}^\dagger T$.
Let $\tilde V = \sqrt{(k-1)/k}\,V$ and define $\tilde U$ analogously, so that these two matrices now contain i.i.d standard Gaussian entries.
Evaluating the moment generating function for the Gaussian random matrix $\tilde U$ and invoking the standard linear algebra result that $\|AB\|_F \le \|A\|_2 \|B\|_F$, it follows that
\begin{align*}
	\Exp_{v,u} \exp\!\left(\frac{\alpha \rho}{k}\langle U,V(\omega - \allones/k) \right)
	&= \Exp_{v,u} \exp\!\left(  \frac{\alpha\rho}{k-1}   \langle \tilde{U}, \tilde{V} (\omega - \allones/k ) \rangle \right) \\
	&= \Exp_{v} \exp\!\left(\frac{\alpha^2\rho^2}{2(k-1)^2}\left\|\tilde V (\omega - \allones/k)\right\|^2_{F}\right) \\
	&\le \Exp_{v} \exp\!\left(\frac{\alpha^2\rho^2}{2(k-1)^2} \| \tilde V \|_2^2 \norm{(\omega - \allones/k)}^2_{F}\right).
\end{align*}
In view of~\cite[Corollary 5.35]{vershynin2010nonasym}, 
$\tilde V$ is a `tall' Gaussian random matrix with spectral norm tightly concentrated:
\[
	\Pr\left[\|\tilde{V} \|_2 \le \sqrt{n} + \sqrt{k} + \epsilon \right] \le \,\e^{-\epsilon^2/2}.
\]
Since $k$ is assumed to be fixed constant, with high probability $\|\tilde{V} \|_2 \le  \sqrt{(1+\epsilon)n}$ for an arbitrarily small constant $\epsilon>0$. 
Notice that if this event does not hold, the second moment becomes unbounded. 
We therefore will compute the second moment conditioned on the high probability event that $\| \tilde V \|$ is not abnormally large.
As we discussed in \S3.2, we define the event 
 $$
 \calF=\{ V: \|\tilde{V} \|_2 \le  \sqrt{(1+\epsilon)n}  \},
 $$
 and conditional distributions $\P'(v, \sigma) = \P (v, \sigma) \indicator{\calF}  / \P(\calF)$ and $\P'(X) =\Exp_{v,\sigma \sim \P'} [\P(X |v, \sigma) ].$
 Notice that $\prob{\calF} \to 1.$
Then the conditional second moment satisfies 
 \begin{align*}
 \Exp_{X\sim Q}\left( \frac{\P'(X)}{\Q(X)} \right)^2  &\le \frac{1}{\P^2(\calF)}  
 \Exp_{\sigma,\tau} \Exp_{v,u} \exp\!\left(\frac{\alpha  \rho}{k}\langle U,V(\omega - \allones/k) \right) \indicator{ U \in \calF} \indicator {V \in \calF} \\
 & \le  \frac{1}{\P^2(\calF)} \Exp_{\sigma,\tau} \Exp_{v,u} \exp\!\left(\frac{\alpha  \rho}{k}\langle U,V(\omega - \allones/k) \right)  \indicator {V \in \calF} \\
 & \le   \frac{1}{\P^2(\calF)}  \Exp_{\sigma,\tau}  \Exp_{v} \exp\!\left(\frac{\alpha^2\rho^2}{2(k-1)^2} \| \tilde V \|_2^2 \norm{(\omega - \allones/k)}^2_{F}\right) \indicator {V \in \calF} \\
 & \le   \frac{1}{\P^2(\calF)}  \Exp_{\sigma,\tau} \exp\!\left( (1+\epsilon) \frac{\alpha^2\rho^2 n}{2(k-1)^2} \| \omega - \allones/k)\|^2_{F} \right).
 \end{align*}

%

As in the submatrix localization problem, it remains  to show that 
\begin{align} \label{eq: clustering_second_moment_laplace}
	\Exp_{\sigma,\tau}  \exp\!\left( (1+\epsilon) \frac{\alpha^2\rho^2 n}{2(k-1)^2} \| \omega - \allones/k)\|^2_{F} \right)
\end{align}
is bounded by a constant for a sufficiently small $\epsilon$. 
This is guaranteed if
$\alpha \rho^2 <1$  for $k=2$ and $\alpha\rho^2 < 2(k-1)\log(k-1)$ for $k>2$, 
in view of Lemmata \ref{lmm:laplacemethod}, \ref{lmm:A-N} and \ref{lmm:A-N_k_two}.
\\

For reconstruction,
notice that the conditioned planted model $\P'$ is still an additive Gaussian model. 
Let $\P'(M |X)$ denote the posterior distribution of $M$ given $X$ under $\P'$, \ie, $\P'(M|X)= \P( X|M) \P'(M)/ \P'(X)$. 
Let $E'(X) =  \Exp_{M \sim \P'(M|X) } [M]$ denote the posterior mean under $\P'$. 
Applying~\prettyref{thm:MMSE} with $\P'$ and $\Q$, we obtain 
$$
\lim_{n\to\infty} \frac{1}{n}  \Exp_{M, X \sim \P' } \left \|  M -  E'(X) \right \|_F^2 
= \lim_{n\to\infty} \frac{1}{n} \Exp_{M \sim \P' } \left \|  M \right \|_F^2 
= \lim_{n\to\infty} \frac{1}{n \P(\calF) } \Exp_{M \sim \P }  \left[ \left \|  M  \right \|_F^2 \indicator{\calF} \right],
$$
where  the last equality follows from the definition of $\P'.$
Let $E(X) =  \Exp_{M \sim \P(M|X) } [M]$ denote the posterior mean under $\P$.
Then
\begin{align*}
  \Exp_{M, X \sim \P' } \left \|  M -  E'(X) \right \|_F^2  & \le     \Exp_{M, X \sim \P' } \left \|  M -  E(X) \right \|_F^2   \\
  & = \frac{1}{\P(\calF)}  \Exp_{M, X \sim \P } \left[  \left \|  M -  E(X)  \right \|_F^2    \indicator{\calF} \right]\\ 
  & \le \frac{1}{\P(\calF)}  \Exp_{M, X \sim \P } \left \|  M -  E(X)  \right \|_F^2 ,
\end{align*}
where the first inequality holds because $E'(X)$ minimizes the  mean squared error under $\P'$;
the second equality holds by the definition of $\P'$. 
Combining the last two displayed equations yields that 
\begin{align}
\liminf_{n\to\infty} \frac{1}{n}  \Exp_{M, X \sim \P} \left \|  M -  E(X) \right \|_F^2  \ge  
\lim_{n\to\infty} \frac{1}{n  } \Exp_{M \sim \P }  \left[ \left \|  M  \right \|_F^2 \indicator{\calF} \right].
\label{eq:gaussian_clustering_conditional_reconstruction}
\end{align}
Since 
$$
\|M\|_F = \frac{\rho}{n} \| (S- \allones_{m,k}/k) V \|_F \le \frac{\rho}{n} \|S-\allones_{m,k}\|_F \|V\|_F = O( \|V\|_F),
$$
it follows that 
$$
 \Exp_{M \sim \P } \left[ \left \|  M  \right \|_F^4 \right] = O\left( \Exp_{V} [ \|V\|_F^4] \right) = O(n^2).
$$
Hence,
$$
\Exp_{M \sim \P } \left[ \left \|  M  \right \|_F^2 \indicator{\calF^c} \right]
\le \sqrt{ \Exp_{M \sim \P } \left[ \left \|  M  \right \|_F^4 \right]  \P(\calF^c) }
= o\left( \sqrt{ \Exp_{M \sim \P } \left[ \left \|  M  \right \|_F^4 \right] } \right)
=o(n).
$$
Combing the last displayed equation with \prettyref{eq:gaussian_clustering_conditional_reconstruction}
gives that 
$$
\liminf_{n\to\infty} \frac{1}{n}  \Exp_{M, X \sim \P} \left \|  M -  E(X) \right \|_F^2  \ge  
\lim_{n\to\infty} \frac{1}{n  } \Exp_{M \sim \P }  \left[ \left \|  M  \right \|_F^2 \right].
$$
Since $E(X)$ minimizes the mean squared error under $\P$, $ \Exp_{M, X \sim \P} \left \|  M -  E(X) \right \|_F^2
\le \Exp_{M \sim \P} \left \| M  \right \|_F^2$,
it follows that 
$$
\lim_{n\to\infty} \frac{1}{n}  \Exp_{M, X \sim \P} \left \|  M -  E(X) \right \|_F^2 =
\lim_{n\to\infty} \frac{1}{n  } \Exp_{M \sim \P }  \left[ \left \|  M  \right \|_F^2 \right],
$$
and thus $ \lim_{n\to \infty} (1/n) \Exp_{X \sim \P} \| E(X) \|_F^2 = 0$. 
By \prettyref{eq:overlap_general} in \prettyref{thm:MMSE}, we conclude that 
for any estimator $\hat{M}$ with $\|\hat{M}\|_F^2 = O(n)$,  $\Exp [\inner{M}{\hat{M}} ] = o(n)$. 


\subsection{Proof of Theorem~\ref{thm:MMSE}}
We give the proof for i.i.d.\ Gaussian noise, using a type of interpolation argument where we vary the signal to noise ratio; the proof for Wigner noise is identical.  Assume that $X(\beta) = \sqrt{\beta} M+W$ 
in the planted model and $X=W$ in the null model, where $\beta \in [0,1]$ is a signal-to-noise ratio parameter (analogous to an inverse temperature) and $W_{ij} \iiddistr \calN(0, 1)$ for all $i, j$. 

First recall that the Bayes-optimal estimator minimizing the mean squared error is the expectation of the posterior distribution, 
\[
	\hat{M}_{\mathrm{MMSE}} (X) = \expect{ M | X } \, ,
\]
so that the (rescaled) minimum mean squared error is given by
\begin{equation}
\label{eq:mmse-def}
	\MMSE(\beta) = \frac{1}{n} \,\Exp \| M - \expect{M|X} \|_F^2 \, .
\end{equation}
We will start by proving that, for all $\beta \in [0, 1]$, the MMSE tends to that of the trivial estimator $\hat{M} = 0$, 
\begin{align}
	\limsup_{n \to \infty} \MMSE (\beta)
	&=   \beta \lim_{n\to \infty}  \frac{ 1}{n} \,\Exp \|M\|_F^2  . \label{eq:mmse_formula}
\end{align}
Let us compute the mutual information $I(\beta)$ between $M$ and $X$: 
\begin{align}
	I(\beta) 
	& = \Exp_{M, X} \, \log \frac{\P(X|M)}{\P(X)}
	\nonumber \\
	& =  \Exp_{X} \, \log \frac{\Q(X)}{\P(X)} +  \Exp_{M, X} \, \log  \frac{ \P(X|M)}{ \Q(X) } 
	\nonumber \\
	& = - D_{\rm KL} (\P \| \Q) +  \Exp_{M, X} \left[ \sqrt{\beta } \inner{M}{X} - \frac{ \beta \|M\|_F^2}{2} \right] 
	\nonumber \\
	& = - D_{\rm KL} (\P \| \Q) + \frac{\beta }{2}  \Exp_{M} \, \|M \|_F^2 \, .
	\label{eq:MI_formula0}
\end{align}
By assumption,  we have $D_{\rm KL} (\P \| \Q) = o(n)$.
 This holds for $\beta=1$; by the data processing inequality for KL divergence~\cite[Theorem 2.2]{PW-it}, this holds for all $\beta < 1$ as well.  Thus~\prettyref{eq:MI_formula0} becomes
\begin{align}
	\lim_{n \to \infty} \frac{1}{n} I(\beta ) = \frac{\beta }{2} \lim_{n\to \infty}  \frac{ 1}{n} \Exp\, \|M\|_F^2 \, .
	\label{eq:MI_formula}
\end{align}

Next we compute the MMSE.  Recall the I-MMSE formula~\cite{GuoShamaiVerdu05}, 
\begin{equation}
\frac{\diff I(\beta)}{\diff \beta} = \frac{n}{2} \MMSE(\beta) \, ,
\label{eq:i-mmse}
\end{equation}
which can also be viewed as a classic formula in thermodynamics.  Note that the MMSE is by definition bounded above by the squared error of the trivial estimator $\hat{M} = 0$, so that for all $\beta$ we have
\begin{equation}
 \MMSE(\beta) \le \frac{1}{n} \,\Exp \, \| M \|_F^2 \, .
 \label{eq:trivial}
\end{equation}
Combining these we have
\begin{align*}
	\frac{1}{2} \lim_{n \to \infty} \frac{1}{n} \,\Exp\!\left[ \| M \|_F^2 \right]
	& \overset{(a)}{=} \lim_{n \to \infty} \frac{1}{n} I(1)  \\
	& \overset{(b)}{=} \frac{1}{2}  \lim_{n \to \infty} \int_{0}^{1}  \MMSE(\beta) \,\diff \beta \\
	& \overset{(c)}{ \le} \frac{1}{2} \int_{0}^{1}  \limsup_{n\to \infty}   \MMSE(\beta) \,\diff \beta \\
	& \overset{(d)}{\le}  \frac{1}{2} \int_{0}^{1}   \lim_{n\to \infty}  \frac{1}{n} \,\Exp \, \| M \|_F^2 \,\diff t \\
	& = \frac{1}{2}  \lim_{n \to \infty} \frac{1}{n} \,\Exp \, \| M \|_F^2 \, , 
\end{align*}
where $(a)$ and $(b)$ hold due to~\prettyref{eq:MI_formula} and~\prettyref{eq:i-mmse}, $(c)$ follows from the Fatou lemma, and $(d)$ follows from~\prettyref{eq:trivial}.  Since we began and ended with the same expression, these inequalities must all be equalities.  In particular, since $(c)$ holds with equality, we have
\begin{equation}
	\limsup_{n\to \infty}  \MMSE(\beta) = \lim_{n\to \infty} \frac{1}{n} \,\Exp \, \| M \|_F^2
	\label{eq:c-eq}
\end{equation}
for almost all $\beta \in [0,1]$.  Since $\MMSE(\beta)$ is a non-increasing function of $\beta$, its limit $\limsup_{n \to \infty}   \MMSE(\beta)$ is also non-increasing in $\beta$.  Therefore, \prettyref{eq:c-eq} holds for all $\beta \in [0, 1]$.  This completes the proof of our claim that the optimal estimator has the same asymptotic MMSE as the trivial one.

To show that the optimal estimator actually converges to the trivial one, we expand the definition of $\MMSE(\beta)$ in~\prettyref{eq:mmse-def} and subtract~\eqref{eq:c-eq} from it.  This gives
\begin{equation}
\limsup_{n \to \infty} \frac{1}{n} \expect{ - 2\inner{M}{\expect{M|X} } + \| \expect{M|X} \|_F^2} = 0 \, . 
\label{eq:mmse-2}
\end{equation}

Note that $\Exp \inner{M}{\expect{M|X} }$  is the expected inner product between the ground truth and a draw from the posterior.  By the Nishimori identity~\cite{iba1999nishimori} or the tower property of conditional expectation, this is equal to the expected inner product between two independent draws from the posterior.  By linearity of the inner product, this gives
$$
	\Exp \, \inner{M}{\expect{M|X} }
 	= \Exp \, \inner{\expect{M|X}}{\expect{M|X}}
 	= \Exp \, \| \expect{ M|X} \|_F^2 \, ,
$$
and combining this with~\prettyref{eq:mmse-2} gives (where $\limsup$ becomes $\liminf$ because of a sign change)
\begin{equation}
	\liminf_{n \to \infty} \frac{1}{n} \Exp \, \| \expect{M|X} \|_F^2 = 0 \, .
\end{equation}
Furthermore, for any estimator $\hat{M}=\hat{M}(X)$
 such that $\Exp_{X} \, \| \hat{M}\|^2_F   =O (n)$, 
$$
	\Exp_{M, X} \inner{M}{\hat{M}} = \Exp_{X} \, \inner{ \expect{M|X} }{\hat{M}}  \le   \Exp_{X} \left[ \| \expect{M|X} \|_F \| \hat{M}\|_F  \right]
	\le  \sqrt{ \Exp_{X} [ \| \expect{M|X} \|^2_F ] \, \Exp_{X} [\| \hat{M}\|^2_F ] }
$$
and thus the desired~\prettyref{eq:overlap_general} holds.

%

Finally, we note that if $M$ is of the form $(1/\sqrt n) UV^\dagger$ where the rows of $U$ and $V$ are independently and identically distributed according to some priors, then $\lim_{n\to \infty} \MMSE(\beta)$ exists by~\cite[Proposition III.2]{MontanariPCA14}, and the liminf in~\prettyref{eq:conditional_mean} and~\prettyref{eq:overlap_general} can be replaced by lim.  Then the MMSE estimator $\hat{M}_{\mathrm{MMSE}} = \expect{M|X}$ tends to the trivial estimator $\hat{M}=0$ as claimed.  

\subsection{Proof of \prettyref{thm:conditional_KL}}
We give the proof of i.i.d.\ Gaussian noise; the proof for Wigner noise is identical. 
Let
$$
Z(X):=\frac{\P (X) }{\Q (X)} 
=\Exp_{M} \left[ e^{ \langle M , X \rangle - \|M \|_F^2/2 } \right].
$$
By the definition of KL divergence, 
\begin{align*}
& D_{\mathrm{KL}} (\P' \| \Q)- D_{\mathrm{KL}} (\P \| \Q) \\
& = \Exp_{X \sim \P' }  \left[ \log \frac{ \P' (X) }{ \Q(X) } \right]  - \Exp_{X \sim \P}   
\left[  \log \frac{\P(X) }{ \Q(X) } \right]  \\
& = \Exp_{X \sim \P' }  \left[ \log \frac{ \P' (X) }{ \P(X) } \right]  + \Exp_{X \sim \P' } \left[  \log \frac{\P(X) }{ \Q(X) } \right] - \Exp_{X \sim \P}   
\left[  \log \frac{\P(X) }{ \Q(X) } \right] \\
& \ge \Exp_{X \sim \P' } \left[  \log Z(X) \right] -  \Exp_{X \sim \P}   
\left[  \log Z(X) \right],
\end{align*}
where the last inequality follows because 
$D_{\mathrm{KL}} (\P' \| \P) = \Exp_{X \sim \P' }  \left[ \log ( \P'(X) / \P(X) ) \right] \ge 0$.  
%
Furthermore,
\begin{align*}
&  \Exp_{X \sim \P' } \left[  \log Z(X) \right] -  \Exp_{X \sim \P}   
\left[  \log Z(X) \right] \\
& =\Exp_{\P_M} \frac{1}{\P(F_M |M) }  \Exp_{\P( X|M ) }  \left[ \left( \indicator{F_M} - \P(F_M|M) \right) \log Z(X)  \right]  \\
 &  \ge   -  \Exp_{\P_M} \frac{1}{\P(F_M |M) }  \sqrt{  \P(F_M |M) \left(1- \P(F_M |M) \right) 
  \Exp_{\P( X|M ) }  \left[  \log^2  Z(X) \right]} \\
 &  = - o(1) \times  \Exp_{\P_M} \sqrt{\Exp_{\P( X|M ) }  \left[  \log^2  Z(X) \right] },
 \end{align*}
 where we use the  Cauchy-Schwarz inequality in the third line, and the assumption that 
 $\P(F_M |M)  =1+o(1)$ uniformly over $M$ in the last line.

By the assumptions that  $\|M\|_2=O(\sqrt{n})$ and $\|M\|_\ast=O(\sqrt{n})$ uniformly over all $M$,
and in view of $\langle M, X \rangle \le \|M\|_\ast \|X\|_2$ and $ \|M\|_F^2 \le \|M\|_\ast \|M\|_2$, 
we have that $ | \log Z(X) | \le O \left( \sqrt{n} \|X\|_2 \right) +O(n).$
 Therefore, 
 $$
 \Exp_{\P( X|M ) }  \left[  \log^2  Z(X) \right] = O(n^2) + O(n) \times \Exp_{\P( X|M ) }  \left[  \|X\|_2^2 \right]
 =O(n^2) + O(n) \times \Exp_{W }  \left[  \|M +W\|_2^2 \right] = O(n^2),
 $$ 
 where the last inequality holds because  $\Exp_{W }  \left[  \|M +W\|_2^2 \right] \le 2 \|M\|_2^2 + 2\Exp_{W} \left[ \|W\|_2^2 \right] $
 and $\Exp_{W} \left[ \|W\|_2^2 \right] = O(n).$
Combining  the last three displayed equations together yields that 
 $$
 D_{\mathrm{KL}} (\P' \| \Q)- D_{\mathrm{KL}} (\P \| \Q) 
 \ge \Exp_{X \sim \P' } \left[  \log Z(X) \right] -  \Exp_{X \sim \P}   
\left[  \log Z(X) \right] \ge - o(n).
 $$

\section*{Acknowledgments}
We thank Yihong Wu, Lenka Zdeborov{\'a}, Florent Krzakala, Thibault Lesieur, Caterina de Bacco, Alex Russell, Andrea Montanari, and Cosma Shalizi for helpful discussions. C. Moore and J. Banks are supported by the John Templeton Foundation and the ARO under contract W911NF-12-R-0012. J. Xu is supported by Simons-Berkeley Research Fellowship. R. Vershynin is suppored by NSF Grant DMS 1265782 and U.S. Air Force Grant FA9550-14-1-0009.  Much of this work was done while the authors were visiting the Simons Institute for the Theory of Computing.

\bibliographystyle{abbrv}
\bibliography{sparse_pca}

\end{document}